\documentclass{amsart}

\usepackage{amsmath,amssymb,amsfonts,mathrsfs}
\usepackage{amsxtra}
\usepackage{enumitem}
\usepackage{mathtools}
\usepackage{stmaryrd}
\usepackage[all]{xy}
\usepackage[colorlinks=true]{hyperref}
\usepackage[dvipdfmx]{graphicx}
\usepackage{xcolor}
\usepackage{bmpsize}

\theoremstyle{plain}
\newtheorem{theorem}{Theorem}[section]
\newtheorem{lemma}[theorem]{Lemma}
\newtheorem{proposition}[theorem]{Proposition}
\newtheorem{corollary}[theorem]{Corollary}
\theoremstyle{definition}
\newtheorem{definition}{Definition}[section]
\newtheorem{remark}{Remark}[section]

\newcommand{\IC}{\mathbb C}

\newcommand{\IN}{\mathbb N}
\newcommand{\IP}{\mathbb P}
\newcommand{\IQ}{\mathbb Q}

\newcommand{\IZ}{\mathbb Z}

\newcommand{\bx}{\boldsymbol{x}}
\newcommand{\by}{\boldsymbol{y}}
\newcommand{\bz}{\boldsymbol{z}}
\newcommand{\bw}{\boldsymbol{w}}

\newcommand{\cB}{\mathcal{B}}
\newcommand{\cC}{\mathcal{C}}
\newcommand{\cD}{\mathcal{D}}
\newcommand{\cE}{\mathcal{E}}

\newcommand{\cI}{\mathcal{I}}

\newcommand{\cL}{\mathcal{L}}
\newcommand{\cM}{\mathcal{M}}

\newcommand{\cO}{\mathcal{O}}

\newcommand{\cU}{\mathcal{U}}
\newcommand{\cV}{\mathcal{V}}

\newcommand{\cZ}{\mathcal{Z}}

\newcommand{\kS}{\mathfrak{S}}

\newcommand{\greg}{\mathscr{G}}

\newcommand{\marc}{\mathscr{M}}

\newcommand{\HH}{{\rm H}}

\renewcommand{\Tilde}{\widetilde}
\renewcommand{\Hat}{\widehat}

\newcommand{\ie}{{\it i.e.\ }}

\newcommand{\Jac}[1]{{\rm Jac}(#1)}
\newcommand{\NS}[1]{{\rm NS}(#1)}
\newcommand{\Pic}[2]{{\pic}^{#1}(#2)}
\newcommand{\Sym}[2]{{\sym}^{#1}(#2)}
\newcommand{\Symo}[2]{{\sym}_{0}^{#1}(#2)}
\newcommand{\Hilb}[2]{{\hilb}^{#1}(#2)}
\newcommand{\Kum}[2]{{\kum}^{#1}(#2)}

\newcommand{\veps}{\varepsilon}

\DeclareMathOperator{\abel}{AJ}
\newcommand{\bl}{\beta}
\DeclareMathOperator{\rest}{res}

\DeclareMathOperator{\coker}{coker}
\DeclareMathOperator{\Aut}{Aut}
\DeclareMathOperator{\Span}{Span}
\DeclareMathOperator{\rank}{rk}
\DeclareMathOperator{\Grass}{Grass}

\DeclareMathOperator{\Discr}{Discr}
\DeclareMathOperator{\Res}{Res}
\DeclareMathOperator{\Supp}{Supp}

\DeclareMathOperator{\Bl}{Bl}
\DeclareMathOperator{\red}{red}
\DeclareMathOperator{\id}{id}
\DeclareMathOperator{\pr}{pr}
\DeclareMathOperator{\hilb}{Hilb}
\DeclareMathOperator{\kum}{Kum}
\DeclareMathOperator{\sym}{Sym}
\DeclareMathOperator{\pic}{Pic}
\DeclareMathOperator{\OGsix}{OG6}
\DeclareMathOperator{\OGten}{OG10}

\newcommand{\KK}{{\rm K3}}
\newcommand{\ogsix}{{\OGsix}}
\newcommand{\ogten}{{\OGten}}

\newcommand{\Ext}{\mathchoice{{\textstyle\bigwedge}}
	{{\bigwedge}}
	{{\textstyle\wedge}}
	{{\scriptstyle\wedge}}}

\newcommand{\Macaulay}{{\bf Macaulay2~}}

\newcommand{\Magma}{{\bf Magma~}}

\renewcommand{\AA}{\mathbb A}

\newcommand{\CC}{\mathbb C}

\newcommand{\PP}{\mathbb P}

\newcommand{\latt}[1]{{\langle{#1}\rangle}}
\newcommand{\inj}{\hookrightarrow}
\newcommand{\To}{\longrightarrow}

\newcommand{\SerreR}[1]{\mathrm{R}_{#1}}
\newcommand{\SerreS}[1]{\mathrm{S}_{#1}}



\begin{document}

\title[On the linear system of cubics of a genus two curve]{On the generalised Kummer fourfold of the Jacobian of a genus two curve}

\author{Samuel Boissi\`ere}

\address{Samuel Boissi\`ere,
	Laboratoire de Math\'ematiques et Applications,
	UMR 7348 du CNRS,
	B\^atiment H3,
	Boulevard Marie et Pierre Curie,
	Site du Futuroscope,
	TSA 61125,
	86073 Poitiers Cedex 9,
	France}
\email{samuel.boissiere@univ-poitiers.fr}
\urladdr{http://www-math.sp2mi.univ-poitiers.fr/{\char'176}sboissie/}

\author{Marc Nieper--Wi{\ss}kirchen}

\address{Marc Nieper-Wi{\ss}kirchen,
	Lehrstuhl f\"ur Algebra und Zahlentheorie,
	Universit\"ats-stra{\ss}e 14, 86159 Augsburg,
	Germany}
\email{marc.nieper-wisskirchen@math.uni-augsburg.de}
\urladdr{http://www.math.uni-augsburg.de/alg/}

\author{Gregory Sankaran}

\address{Gregory Sankaran,
University of Bath,
Bath BA2 7AY, England}
\email{G.K.Sankaran@bath.ac.uk}
\urladdr{https://people.bath.ac.uk/masgks/gks.html}

\begin{abstract}
	We construct a birational model of the generalised Kummer
        fourfold of the Jacobian of a genus two curve, based on a
        geometric interpretation of the addition law on this Jacobian,
        obtained by the properties of the linear system of cubics on
        that curve. We show that our model has mild singularities and
        that it admits a finite ramified covering to the
        four-dimensional projective space.
\end{abstract}

\maketitle

\setcounter{tocdepth}{1}
\tableofcontents

\section{Introduction} \label{s:introduction}

An irreducible holomorphic symplectic manifold, abbreviated to IHS
manifold, is a simply-connected compact complex manifold $X$ such that
$\HH^{2,0}(X)$ is generated by an everywhere nondegenerate
$2$-form. If $X$ is an IHS manifold then $\HH^2(X, \IZ)$ naturally
carries a nondegenerate integral quadratic form $q_X$ of signature
$(3, b_2(X)-3)$ (see~\cite{Beauville}), the
\emph{Beauville--Bogomolov--Fujiki quadratic form}.  We denote by
$\latt{ -, -}_X$ the associated bilinear form. In all the known
deformation types, the lattice $\left(\HH^2(X, \IZ), q_X\right)$ is
even, but not unimodular, except when $X$ is a K3 surface.

All the IHS manifolds that we consider will furthermore be projective.
A \emph{polarisation} of $X$ is a primitive ample class $L$ in the
N\'eron--Severi group $\NS X $ of $X$. The \emph{degree} of the
polarisation is the positive integer $d\coloneqq q_X(L)$ and its
\emph{divisibility} is the integer $\gamma$ such that $\latt{ \HH^2(X,
  \IZ), L}_X = \gamma \IZ$.

Gritsentko, Hulek and Sankaran~\cite{GHS_moduliIHS} constructed coarse
moduli spaces of polarised pairs $(X, L)$ of a given deformation type:
these moduli spaces are quasi-projective varieties. Our initial
motivation in this paper is to search for concrete geometric
descriptions of the generic elements in some of these moduli spaces
$\cM$. In practice, in most cases, such a description of a generic object can be
used to construct a dominant rational map $\IP^N \dashrightarrow \cM$
for some integer~$N$, so that $\cM$ would be unirational.

In general it is hard to decide whether a given moduli space $\cM$ is
unirational or not. The general philosophy is that these moduli spaces
may be unirational for low values of the numerical invariants but will
be of general type when the numerical invariants are high. For example
this is the case for K3 surfaces (see~\cite{GHS_moduliK3}) and for any
putative class of IHS manifolds whose moduli space is of large
dimension (see~\cite{Ma}), and analogous statements hold for
moduli of curves and of abelian varieties.
Instead of unirationality
one could ask for related properties such as being rational (stronger)
or uniruled, stably rational or rationally connected (weaker). In
Appendix~\ref{app:moduli_spaces} we summarise the currently known
results about the birational types of moduli spaces of polarised IHS
manifolds (see also~\cite{Debarre_Milan}).

Except for K3 surfaces, each of the known moduli spaces is named after
a codimension one family. For instance, the four-dimensional moduli
spaces $\cM_{\kum^2}^{d, \gamma}$ that feature in this paper
parametrise polarised IHS manifolds of Kummer type, \ie deformation
equivalent to the second generalised Kummer variety of an abelian
surface, of degree $d$ and divisibility $\gamma$.
Most of the unirationality results for moduli of polarised IHS
manifolds concern the deformation class of Hilbert type. For the other
known types, the question is relatively unexplored, apart from the
recent results of Barros, Beri, Flapan and Williams~\cite{BBFW} for
the generalised Kummer and OG6 cases. In this paper, we focus on the
deformation type of the second generalised Kummer variety of a
polarised abelian surface.  In order to attack the unirationality
question in this deformation class, our first objective, which we
achieve in the present paper, is to construct and to study a
birational model of a generalised Kummer fourfold using only rational
tools.  Our second objective, which is still work in progress, will
then be to understand how this construction may deform. Benedetti,
Manivel and Tanturri~\cite{BMT} worked on a similar question, from a
different point of view, using Coble hypersurfaces to get models of
generalised Kummer fourfolds as flag varieties, but their construction
does not deform. Very recently another construction of a similar
nature has been given by Agostini, Beri, F.~Giovenzana and R\'{\i}os
Ortiz in~\cite{ABGRO}. 

By the general philosophy on the moduli spaces, we guess that to lower
the discrete invariants it is wise to lower the polarisation.  We
therefore consider principally polarised abelian surfaces, and study
the codimension~$1$ family of generalised Kummer fourfolds over
Jacobians of a genus two curve. It might also be interesting to study
the other principally polarised case, products of elliptic curves,
which will give a codimension 2 family.  

Our original intuition is to fix a genus~$2$ curve $C$ and look at
projective coordinates on $\Jac C$ in a model where addition is well
described. Such models are used in cryptography, for instance by
Flynn~\cite{Flynn_law} and Leitenberger~\cite{Leitenberger}, whose
works on the addition law inspired the present paper. Our main results
are:

\begin{theorem}[Corollary~\ref{cor:main1}]
Let $C$ be a smooth genus two curve. The linear
system of cubics embeds $C$ in $(\IP^4)^\vee$ and the dual variety
$C^\ast\subset \IP^4$ of $C$ is a degree~$14$ irreducible
hypersurface. The second generalised Kummer variety $\Kum 2 {\Jac C}$
of the Jacobian of $C$ is birational to a degree $15$ covering
of~$\IP^4$ branched along $C^\ast$.
\end{theorem}

We denote by $\greg_C$ the degree $15$ covering of $\IP^4$ branched
along $C^\ast$ mentioned in the above statement, whose definition is
given in Definition~\ref{def:greg} and~\S\ref{ss:def_G}, and by:
\[
\gamma_C\colon \greg_C\dashrightarrow \Kum 2 {\Jac C}
\]
the birational map in question, whose definition is given in
Formula~\eqref{eq:gamma}.

\begin{proposition}[Proposition~\ref{prop:multiple_4plane}]
  The variety $\greg_C$ is normal and  Gorenstein, and
  with quotient singularities.
\end{proposition}

In particular, $\greg_C$ is Cohen-Macaulay.

\begin{proposition}[Propositions~\ref{prop:birational_F1} and \ref{prop:birational_F2}]
The birational map $\gamma_C$ contracts one divisor to the
noncurvilinear point of $\Kum 2 {\Jac C}$ supported at the origin of~$\Jac C$, and a second divisor to the
Kummer surface $\Kum 1 {\Jac C}$.
\end{proposition}

\begin{proposition}[Proposition~\ref{prop:marc}]
The Galois closure of the covering $\greg_C\to\IP^4$ is a local
complete intersection scheme.
\end{proposition}

In this paper, the term ``variety'' denotes an integral separated
noetherian scheme of finite type over the field of complex
numbers. The term ``curve'' means an irreducible projective variety of
dimension~$1$.

In~\S\ref{s:general_construction} we construct the
variety~$\greg_C$ under more general assumptions, starting from any
abelian surface $A$. We study the contraction to the generalised
Kummmer fourfold $\Kum 2 A$ in this general setup. Then
in~\S\ref{s:genustwo} we specialise to the case where~$A$ is the
Jacobian of a genus two curve and we study the properties of the
linear system of cubics on the curve. We apply this geometry
in~\S\ref{s:covering} to realise $\greg_C$ as a finite cover of
$\IP^4$. Finally in~\S\ref{s:branch} we study the branch locus of this
cover. Several Appendices contain some backgrounds, alternative or
complementary views of some results proven in the main text,
as such as some helpful computer algebra scripts.

The authors warmly thank Daniele Agostini, Pietro Beri, Enrica Floris, Christian Lehn,
Emanuele Macr\`i, Gianluca Pacienza, Mattieu Romagny, Alessandra Sarti
and Calla Tschanz for helpful discussions. We also thank the
organisers and the participants of the conferences \textit{Around
  Symmetries of K3 Surfaces} at BIRS in 2023 and \textit{Kummers in
  Krakow} in 2024. The first author has been partially funded by the
ANR/DFG project ANR-23-CE40-0026 ``Positivity on K-trivial varieties".

\section{A rational contraction to the generalised Kummer fourfold}
\label{s:general_construction}

We consider a polarised abelian surface $(A, c_1(H))$, with origin
$O_A\in A$, where $H\in \Pic {2t} A$ is an ample divisor of degree
$H^2=2t$ with $t\geq 1$, whose first Chern class $c_1(H)$ is primitive
in $\HH^2(A, \IZ)$.

\subsection{The second generalised Kummer variety of an abelian surface} \label{ss:intro_genKm}

For any integer $m\geq 0$, we denote the Hilbert scheme of
$0$-dimensional subschemes of $A$ of length $m$ by $\Hilb {m} A$,
the Chow quotient by $\Sym {m} A$, and the Hilbert--Chow morphism
by:
\[
h_A\colon \Hilb {m} A \to \Sym {m} A.
\]
The addition law on $A$ defines the following morphisms:

\begin{equation}
	\label{eq:diag_addition_quotient}
	\xymatrix{
		A^{m} \ar[r]^{\alpha_A}\ar[d]_{\pi_A}& A \\
		\Sym {m} A \ar[ru]_{\bar\alpha_A}
	}
\end{equation}
We restrict to the case $m = n + 1$ for an $n \geq 0$ and denote by
 $\Symo {n + 1} A\coloneqq \alpha_A^{-1}(\{O_A\})$ the fibre
over the origin of the addition map $\alpha_A$ and by $\Kum n A$
the \emph{$n$-th generalised Kummer variety of $A$}, defined as the
fibre over the origin of $\alpha_A\circ h_A$:
\[
\Kum n A \coloneqq (\bar\alpha_A\circ h_A)^{-1}(O_A)\subset \Hilb {n+1} A.
\]
The restriction $h_A^\circ$ of the Hilbert--Chow morphism $h_A$ to the
generalised Kummer variety is still birational. It is a resolution of
the singularities of the Chow quotient $\Symo {n+1} A$.  It is well
known \cite{Beauville} that the variety $\Symo {n+1} A$ has
symplectic singularities and that $\Kum n A$ is an irreducible
holomorphic symplectic manifold of dimension $2n$.
The variety $\Kum 1 A$ is the classical Kummer surface associated to~$A$, \ie the minimal resolution of the quotient $A/{\pm 1}$.

In this paper, we are mostly interested in the second
generalised Kummer variety $\Kum 2 A$.  Its second integral
cohomology group decomposes as follows. There exists a natural injection
  $\HH^2(A, \IZ)\inj \HH^2(\Kum 2 A, \IZ)$ and we have:
\[
\HH^2(\Kum 2 A, \IZ)  = \HH^2(A, \IZ)\oplus \IZ \delta,
\]
where $\delta$ is half the class of the exceptional divisor of the
Hilbert--Chow morphism intersected with $\Kum 2 A$. This decomposition
is orthogonal with respect to the lattice structure on $\HH^2(\Kum 2 A, \IZ)$
given by the Beauville--Bogomolov--Fujiki (BBF) form, and the isometry
class of the lattice is computed in~\cite{Rapagnetta}:
\[
\HH^2(\Kum 2 A, \IZ) \cong U^{\oplus 3}\oplus \latt{-6},
\]
where $U$ is the hyperbolic plane.

By the decomposition above, we have a splitting of the N\'eron--Severi
lattice:
\[
\NS {{\Kum 2 A}} = \NS A \oplus \IZ\delta.
\]

We denote by $h\in \NS{\Kum 2 A }$ the image of $c_1(H)$, which is a
big and nef divisor. By a result of Debarre and
Macr\`{\i}~\cite[Corollary~4.11]{DM} the classes $ah-b\delta$, with
$a,\,b>0$, are ample when $b/a<1/3$. If $(A, H)$ is \emph{generic},
meaning that $\NS A = \IZ c_1(H)$, we have $\NS {\Kum 2 A} = \IZ h
\oplus \IZ\delta$.  Furthermore, A.~Mori~\cite{Mori}, has shown that
if $H$ is a principal polarisation, \ie $t= 1$, then the ample cone is
precisely the interior of the cone generated by the classes $h$ and
$2h-\delta$. The smallest possible polarisation degree with respect to
the Beauville--Bogomolov--Fujiki quadratic form is thus given by the
smallest integer $d$ such that $\ell\coloneqq ah-b\delta$ is an ample
class with $q_{\Kum 2 A}(\ell) = d = 2e$.

The smallest integer $e$ such that $e = a^2 -3b^2$ with $a, b\in\IN$
and $b/a<1/2$ is $e =6$, obtained for $(a, b) = (3, 1)$, so the
minimal polarisation is $\ell = 3 h - \delta$, of degree $d = 12$. It
is easy to check that $\latt{ \NS {\Kum 2 A}, \ell} = 6\IZ$. But since
the embedding of $\NS {\Kum 2 A}$ in $\HH^2(\Kum 2 A, \IZ)$ sends
the class $h$ to an element of the unimodular lattice $U^{\oplus 3}$,
there exists $u\in U^{\oplus 3}$ such that $\latt{ u, h} = 1$ and this
implies that $\latt{\HH^2(\Kum 2 A, \IZ), \ell} = 3\IZ$.
The divisibility is thus $\gamma = 3$ and $(\Kum 2 A , \ell)\in
\cM_{\kum^2}^{12, 3}$. This space does not appear in~\cite{BBFW} and
nothing is known about its birational geometry.

\subsection{The birational model} \label{ss:objectofinterest}

Consider the blowup of the origin of $A$:
\[
\bl_A\colon\Tilde A\coloneqq \Bl_{O_A} A \To A,
\]
with exceptional divisor $E_A\coloneqq \bl_A^{-1}(O_A)$.  First,
using the summation maps defined in
Diagram~\eqref{eq:diag_addition_quotient} we put:
\[
A^3_0\coloneqq \alpha_A^{-1} (O_A),
\]
and we denote by $\pi_A^\circ$ the restriction of the Chow quotient:
\[
\xymatrix{
	A^3_0 \ar@{^(->}[r]\ar[d]_{\pi_A^\circ}& A^3\ar[d]_{\pi_A} \ar[r]^{\alpha_A} & A \\
	\Symo 3 A \ar@{^(->}[r]& \Sym 3 A  \ar[ru]_{\bar\alpha_A}
}
\]
We do similarly starting with $\Tilde A$; we define $\Tilde A^3_0$
as the fibre over the origin of the morphism:
\[
\Tilde A^3 \xrightarrow{\bl_A^{\times 3}} A^3 \xrightarrow{\alpha_A} A,
\]
and we finally define the main object of interest in this paper:

\begin{definition}\label{def:greg}
We denote by $\greg_A$ the scheme-theoretic fibre over the origin of
the morphism:
\[
\Sym 3 {\Tilde A} \xrightarrow{\Sym 3 {\bl_A}} \Sym 3 A
\xrightarrow{\bar\alpha_A} A,
\]
that is: $\greg_A\coloneqq \Symo 3 {\Tilde A}\coloneqq
(\bar\alpha_A\circ \Sym 3 {\bl_A})^{-1}(O_A)$.
\end{definition}

The morphism $\Sym 3 {\bl_A}$ is clearly birational and its
restriction to $\greg_A$ is still birational since it is an
isomorphism above the open subset of triples of nonzero points on $A$
whose sum is zero. We are interested in the birational map:
\begin{align}\label{eq:gamma}
\gamma_A\coloneqq h_A^{-1}\circ \left.\Sym 3 {\bl_A}\right|_{\greg_A}\colon \greg_A \dashrightarrow \Kum 2 A.
\end{align}

All the relevant maps are shown in the diagram~\eqref{diagram_total} below:
\begin{equation}\label{diagram_total}
\xymatrix{
	\Tilde A^3\ar[d]_{\pi_{\Tilde A}} & \Tilde A^3_0\ar[d]_{\pi_{\Tilde A}^\circ}\ar[rrr]^{\left.\bl_A^{\times 3}\right|_{\Tilde A^3_0}}\ar@{_(->}[l] &&& A^3_0\ar[d]_{\pi_{A}^\circ} \ar@{^(->}[r]& A^3\ar[d]_{\pi_{A}}\\
	\Sym 3 {\Tilde A} & \greg_A\ar@{_(->}[l]\ar[rrr]^{\left.\Sym 3 {\bl_A}\right|_{\greg_A}}\ar@{-->}[drrr]_{\gamma_A} &&& \Symo 3 A \ar@{^(->}[r] & \Sym 3 A\\
	&&&& \Kum 2 A\ar[u]^{h_A^\circ}\ar@{^(->}[r] & \Hilb 3 A \ar[u]^{h_A} \\
}
\end{equation}

\begin{proposition}\label{prop:GA_normal}
The scheme $\greg_A$ is reduced. It is a normal projective variety of
dimension $4$, Cohen--Macaulay and $\IQ$-factorial with
quotient singularities, and it is birational to $\Kum 2 A$.
\end{proposition}

\begin{proof}
Let $\mathfrak a \to A$ be the natural cover where $\mathfrak a$ is
the (abelian) Lie algebra of~$A$.  By choosing linear coordinates
$(\bx, \by)$ on $\mathfrak a$, we get an identification $\mathfrak a
\to \IC^2$
and thus a cover $\IC^2 \to A$ of groups (in the category
of complex manifolds).  The linear coordinates $(\bx, \by)$ induce
local coordinates on $A$ around each point, up to a choice of an
element in the kernel of $\IC^2 \to A$. Near $O_A$, we always make
this choice such that the coordinates of $O_A$ become $(0, 0)$.
Around a point $a\in A$, with
$a\neq O_A$, the coordinates $(\bx, \by)$ are also local coordinates of $\Tilde A$ at the
point $\tilde a\coloneqq\bl_A^{-1}(a)$. Around the point $a= O_A$, local coordinates
of $\Tilde A$ near the exceptional divisor~$E_A=\bl_A^{-1}(O_A)$
are $\bx, \bz$ with relation $\by = \bx \bz$ (or with the roles of
$\bx$ and $\by$ exchanged).  The map $\bl_A$ is then locally given by
$\bl_A(\bx, \bz) = (\bx, \bx \bz)$.

With this convention, around each point $a\in A$ 
the local coordinates are such that the addition law is the
standard one. That is, around a point $(a_1, a_2, a_3)\in A^3$, the
local coordinates $(\bx_1, \by_1), (\bx_2, \by_2), (\bx_3, \by_3)$
are such that the subvariety $A^3_0$ is locally given by the relations $\bx_1 +
\bx_2 + \bx_3 = 0$ and $\by_1 + \by_2 + \by_3 = 0$.
This allows us to analyse the singularities of $\Tilde A^3_0$ by
computing local equations of $\Tilde A^3_0$ using the morphism
$\bl_A^{\times 3}$.  Again the local equations of $\Tilde A^3_0$ are
$\bx_1 + \bx_2 + \bx_3 = 0$ and $\by_1 + \by_2 + \by_3 = 0$, but
$\by_i$ is in degree~$1$ if $a_i\neq O_A$ and in degree~$2$ if
$a_i=O_A$. Therefore, unless $a_1=a_2=a_3=O_A$, the ideal defining
$\Tilde A^3_0$ in the local ring is generated by two elements with
independent linear parts, so $\Tilde A^3_0$ is smooth. If
$a_1=a_2=a_3$ then $\Tilde A^3_0$ is locally a linear section (by
$\bx_1+\bx_2+\bx_3=0$) through a rank~$3$ quadric cone
$\bx_1\bz_1+\bx_2\bz_2+\bx_3\bz_3=0$, which is again a quadric cone in~$\AA^5$, of rank~$2$. Indeed, a local equation is $\bx_1 \bw_1 + \bx_2
\bw_2 = 0$, in local coordinates $\bx_1,\,\bx_2,\bw_1\coloneqq
\bz_1-\bz_3,\, \bw_2\coloneqq \bz_2-\bz_3,\, \bz_3$.

It follows that $\Tilde A^3_0$ is normal and has hypersurface
singularities, so it is Gorenstein~\cite[Corollary 21.19]{Eisenbud},
and it is connected and irreducible by Zariski's Main Theorem.

Since $\greg_A$ is the quotient of $\Tilde A^3_0$ by the action of
the symmetric group $\kS_3$ acting by permutation of the factors, we
deduce easily that $\greg_A$ is reduced, normal, connected and
irreducible.  Moreover, $\greg_A$ is Cohen-Macaulay by
the Hochster--Roberts theorem~\cite[Main Theorem and Remark~2.3]{HR}
and it is $\IQ$-Gorenstein (see the argument in the proof
of~\cite[Lemma 5.16]{KM}. We can even be more specific here: since
$\Tilde A^3_0$ has transversal nodal singularities, it has in
particular quotient singularities, so $\greg_A$ too. It follows that
$\greg_A$ is $\IQ$-factorial with rational singularities
(see~\cite[Proposition~5.15]{KM}).
\end{proof}

\begin{remark}\label{rem:non_cartier}
The variety $\Tilde A^3_0$ is not locally factorial. Consider for
instance the divisor $F_1=\Sym 3 {E_A}$ and its pre-image
$\Tilde{F}_1\coloneqq (\pi^\circ_{\Tilde A})^{-1} (F_1)$.  In the
local chart used in the proof above, the divisor $E_A$ has equation
$\bx= 0$, so $\Tilde F_1$ has equations $\bx_1 = \bx_2 = 0$ inside
$\Tilde{A}^3_0$: it is not a Cartier divisor. Computations with
Macaulay2~\cite{Macaulay2} indicate that $\Symo 3 {\tilde{A}}$ and
$\greg_A$ are not local complete intersection schemes. They indicate
also that $\greg_A$ is Gorenstein: we will prove it in
Proposition~\ref{prop:multiple_4plane} under the assumption that $A$
is the Jacobian of a genus two curve. The script is given in
Remark~\ref{macaulay:lci}.
\end{remark}

\subsection{The divisorial contraction to the Chow quotient}\label{ss:divisorial_chow}

To fully describe the geometric relation between~$\greg_A$ and~$\Symo
3 A$, we exhibit two meaningful divisors $F_1, F_2$ on $\greg_A$ that
parametrise special configurations of triples of points on $\Tilde
A$.

Recall that $E_A\subset \Tilde A$ is the exceptional divisor of the
blowup $\bl_A\colon \Tilde A \to A$. As in
Remark~\ref{rem:non_cartier} we define the prime
divisor
\begin{align}\label{def:F1}
  F_1\coloneqq \Sym 3 {E_A}\subset \greg_A.
\end{align}
Since $E_A\cong\IP^1$, the divisor~$F_1$ is isomorphic to~$\IP^3$.

Denote by $\tau\colon A\to A$, $a\mapsto -a$ the sign involution and
by $\bar a\in A/\tau$ the class of $a\in A$. There is an embedding of
$A/\tau$ in $\Symo 3 A$ given by $\bar a\mapsto a + (-a) + O_A$. The
surface $A/\tau$ contains in particular the point~$3 \,O_A$.  The
sign involution~$\tau$ on~$A$ lifts to $\Tilde A$ as an involution
denoted $\tilde\tau$, making the blowup morphism $\bl_A\colon \Tilde
A\to A$ equivariant, that is $\bl_A\circ \tilde\tau = \tau\circ
\bl_A$, and leaving the exceptional divisor $E_A$ pointwise fixed. We
then define the second prime divisor $F_2\subset \greg_A$ as the image
of the morphism
\[
\Tilde A/{\tilde\tau} \times E_A \to \greg_A, \quad (\tilde a, e)\mapsto \tilde a + \tilde\tau(\tilde a) + e,
\]
that is:
\begin{align}\label{def:F2}
	F_2\coloneqq\left\{\tilde a + \tilde\tau(\tilde a) + e\mid (\tilde a, e)\in \Tilde A/{\tilde\tau} \times E_A\right\}\subset \greg_A.
\end{align}

\begin{proposition}
The birational morphism $\Sym 3 {\bl_A}\colon \greg_A \to \Symo 3 A$
contracts the divisor $F_1$ to the point $3\, O_A$, and it contracts
the divisor $F_2$ to the surface $A/{\tau}$. It is $1:1$ outside of
these two divisors.
\end{proposition}

\begin{proof}
The divisor $F_1$ is contracted to the point $3 O_A$ since
$\bl_A(E) = O_A$. Similarly, with the same notation as above,
\[
\Sym 3 {\bl_A}(\tilde a + \tilde\tau(\tilde a) + e) = a + (-a)  + O_A,
\]
so the divisor $F_2$ is contracted to the surface $A/\tau$.  Take a
point $a + b +c \in \Sym 3 A$. If none of these points is the origin
of $A$, it has a unique preimage by $\Sym 3 {\bl_A}$. If $c=0$, then
$b = -a$ and the fibre over this point belongs to the divisor
$F_2$. If $b=c=0$, then $a=0$ and the fibre over this point is the
divisor $F_1$. So $\Sym 3 {\bl_A}$ is an isomorphism outside of these
two divisors.
\end{proof}

Note that $F_1$ and $F_2$ intersect along the big diagonal of $\Sym 3
{E_A}$ parametrising $0$-cycles with at least one double point.

\subsection{The rational contraction} \label{ss:contractiontogenKum}

We now study the birational map $\gamma_A\colon \greg_A\dashrightarrow
\Kum 2 A$.  Since the variety $\greg_A$ is normal by
Proposition~\ref{prop:GA_normal} and since $\Kum 2 A$ is a projective
variety, the indeterminacy locus of $\gamma_A$ is a subset of
codimension at least two of $\greg_A$~\cite[Lemma~V.5.1]{Hartshorne}.

We first analyse the behaviour of $\gamma_A$ around the divisor $F_1$.
For this we introduce the $2$-dimensional \emph{Brian\c{c}on variety}
$B_{O_A}^3$ parametrising the locus of nonreduced subschemes in $\Kum
2 A$ supported at $O_A$. As the indeterminacy locus of $\gamma_A$ is
of codimension at least $2$, and since the morphism $\Sym 3 {\bl_A}$
contracts $F_1$ to the point~$3\,O_A$, by restriction we have a
rational map
\[
\left.\gamma_A\right|_{F_1}\colon F_1 \dashrightarrow B^3_{O_A},
\]
so the divisor $F_1$ is contracted by $\gamma_A$. Its rational image
is the \emph{centre} of $F_1$ for $\gamma_A$: our goal is to compute
this centre.

For this, let us recall the geometry of the variety $B_{O_A}^3$,
following~\cite[\S2]{Lehn}.  This depends only on the local geometry
of $A$ near $O_A$, so as in the proof
of Proposition~\ref{prop:GA_normal} and Remark~\ref{rem:non_cartier},
we take local coordinates $\bx, \by$ at the origin $O_A\in A$. This
identifies the tangent space $T_{A,O_A}$ with $\CC^2$, compatibly with
addition: hence for the purpose of local computation
we may replace $A$ by $\CC^2$. The \emph{curvilinear} subschemes supported at the
origin arise as limit points of triples of points that move along a
smooth curve. Their ideals have the form $\latt{ \by + \alpha \bx +
\beta \bx^2, \bx^3}$ or similarly with $\bx$~and~$\by$
exchanged. They form a line bundle over $\IP^1$, where the base~$\IP^1$ parametrises the tangent direction of the curve at the origin
(encoded by the parameter $\alpha$) and the fibre depends on the
parameter~$\beta$ that encodes the curvature of the curve. The
Brian\c{c}on variety $B_{O_A}^3$ is obtained by compactifying this
affine bundle by adding as point at infinity the non-curvilinear
subscheme $Z_\infty$ that arises as the limit of triples of points
going to the origin from three different directions: its ideal is
$I_\infty\coloneqq\latt{ \bx^2, \bx \by, \by^2}$.

\begin{proposition}\label{prop:birational_F1}\text{}
The birational map $\gamma_A$ contracts the divisor $F_1$ to the
point~$Z_\infty$.
\end{proposition}

\begin{proof}
The behaviour of $\gamma_A$ at the divisor $F_1$ is a local property
over a neighbourhood of the origin $O_A$ of $A$ so we can study it by
computing a local model of the variety~$\greg_A$ in the neighbourhood
of the divisor~$F_1$, as we did in the proof of
Proposition~\ref{prop:GA_normal}.  Instead of directly computing
$\left.\gamma_A\right|_{F_1}$, it is equivalent, but more convenient,
to study the composite rational map
\[
g\colon \Tilde A^3_0\xrightarrow{\pi^\circ_{\Tilde A}}
\greg_A\overset{\gamma_A}{\dashrightarrow} \Kum 2 {A} \inj
\Hilb 3 A.
\]
We denote by $\Tilde F_1$ the preimage of $F_1$ in $\Tilde
A^3_0$. The rational map $g$ is defined at the generic point of
$\Tilde F_1$, and we want to compute its image and the indeterminacy
locus of the restriction of $g$ to $\Tilde F_1$.

\textit{The local coordinates}. As in the proof
of Proposition~\ref{prop:GA_normal} and Remark~\ref{rem:non_cartier},
we take local coordinates $\bx, \by$ at the origin $O_A\in A$. This
identifies the tangent space $\mathfrak a = T_{A,O_A}$ with $\CC^2$, compatibly with
addition: hence for the purpose of local computation near $\Tilde F_1$
we may replace $A$ by $\CC^2$, and in particular $\Hilb 3 A$ by $\Hilb 3 {\CC^2}$.
Again as in Proposition~\ref{prop:GA_normal} and
Remark~\ref{rem:non_cartier}, with coordinates $(\bx_1, \by_1)$,
 $(\bx_2, \by_2)$, $(\bx_3, \by_3)$ on $\CC^6$, using the relations $\by_i
= \bx_i \bz_i$ and introducing $\bw_1 \coloneqq \bz_1 - \bz_3$ and
$\bw_2 \coloneqq \bz_2 - \bz_3$ we get down to five variables $\bx_1,
\bx_2, \bw_1, \bw_2, \bz_3$ where $\Tilde A^3_0$ is defined by the
single relation $\bx_1 \bw_1 + \bx_2 \bw_2 = 0$ and $\Tilde F_1$~has
local equations $\bx_1 = \bx_2 = 0$.

\textit{The open covering}. 
In what follows, the notation $I\in \Hilb 3 {\CC^2}$
means that $I\subset \IC[\bx, \by]$ is the ideal of the corresponding
length three subscheme of $\Hilb 3 {\CC^2}$.  Following
Haiman~\cite{Haiman}, the Hilbert scheme $\Hilb 3 {\CC^2}$ is covered by
three affine charts labelled by the partitions of the integer~$3$, as
follows. Let
\[
\cB_{(1, 1, 1)}\coloneqq \{1, \bx, \bx^2\}, \quad \cB_{(2,
  1)}\coloneqq \{1, \bx, \by\}, \quad \cB_{(3)}\coloneqq \{1, \by,
\by^2\},
\]
and for any partition $\mu$ of the integer $3$, define the subset
\[
\cU_\mu\coloneqq\left\{I\in \Hilb 3 {\CC^2}\mid \cB_\mu \text{ spans }
\IC[\bx, \by]/{I}\right\}.
\]
By~\cite[Proposition~2.1]{Haiman}, the subsets $\cU_\mu$ are open
affine subvarieties that cover $\Hilb 3 {\CC^2}$. They cover the
Brian\c{c}on subvariety~$B^3_{O_A}$ of $\Hilb 3 {\CC^2}$ parametrising
length~$3$ subschemes supported at the origin, as follows: the
curvilinear subschemes of the form $\latt{ \by + \alpha \bx + \beta
\bx^2, \bx^3}$ belong to~$\cU_{(1, 1, 1)}$, similarly the
curvilinear subschemes of the form $\latt{ \by + \alpha \bx + \beta
\bx^2, \bx^3}$ belong to~$\cU_{(3)}$, whereas the noncurvilinear
point~$Z_\infty$ belongs to~$\cU_{(2, 1)}$, and since it is unique, we
have $\Hilb 3 {\CC^2}\setminus\left(\cU_{(1, 1, 1)} \cup \cU_{(3)}\right) =
\left\{Z_\infty\right\}$.

\textit{The Hilbert--Chow morphism in coordinates}.
Following~\cite[pp. 210--214]{Haiman}, the coordinate ring of the
chart $\cU_{(1, 1, 1)}$ is $\IC[e_1, e_2, e_3, a_0, a_1, a_2]$, and an
ideal $I_{(e, a)}\in \cU_{(1, 1, 1)}$, with coordinates $(e,
a)\coloneqq (e_1, e_1, e_3, a_0, a_1, a_2)$ is given by
\[
I_{(e, a)}\coloneqq \latt{ \bx^3-e_1\bx^2+e_2\bx-e_3, \by-(a_0+a_1\bx
+ a_2\bx^2)}.
\]
Whenever $e = (e_1, e_2, e_3) = 0$ and $a= (0, a_1, a_2)$, we get an
element $I_{(0, a)}\in B^3_{O_A}$ of the Brian\c{c}on variety. The
meaning of these affine coordinates is that if the zero locus
$\cV(I_{(e, a)})$ consists of three points of $\IC^2$
of coordinates $(\bx_1, \by_1)$, $(\bx_2, \by_2)$ and $(\bx_3, \by_3)$,
repeated with multiplicity, then in this chart we use the Vi\`ete
formula
\begin{align}\label{eq:viete}
\bx^3-e_1\bx^2+e_2\bx-e_3 = \prod_{i=1}^3 (\bx-\bx_i),
\end{align}
so that the coordinates $e_i$ are the elementary symmetric functions
in the variables $\bx_1, \bx_2, \bx_3$, whereas the coordinates $a_i$
are the coefficients of the Lagrange interpolation polynomial
$\phi_a(\bx) = a_0 + a_1 \bx + a_2 \bx^2$ such that $\by_i =
\phi_a(\bx_i)$ for $i=1, 2, 3$. We see that the coordinates $a_0, a_1,
a_2$ are well defined only when the three coordinates $\bx_1, \bx_2,
\bx_3$ are different, that is when $I_{(e, a)}$ is a reduced
subscheme.

The Chow quotient $\Sym 3 {\CC^2}$ is $A^3/\kS_3$, where any element
$\sigma$ of the symmetric group $\kS_3$ acts by $\sigma(\bx_i, \by_i)
= (\bx_{\sigma(i)}, \by_{\sigma(i)})$ for any $i=1, 2, 3$.  The
Hilbert--Chow morphism
\[
h_{|\cU_{(1, 1, 1)}}\colon \cU_{(1, 1, 1)} \to \Sym 3 {\CC^2}
\]
is defined by
\[
h(e, a) = ((\bx_1, \by_1),(\bx_2,\by_2),(\bx_3,\by_3)),
\]
where $\bx_i$ are the roots (not necessarily distinct) of the
polynomial $\bx^3-e_1\bx^2+e_2\bx-e_3$, and $\by_i = \phi_a(\bx_i)$ as
above.

In this chart we can describe the birational inverse map
\[
h^{-1}\colon\Sym 3 {\CC^2}\dashrightarrow \Hilb 3 {\CC^2}:
\]
the coordinates $(e_1, e_2, e_3)$ are always defined by
formula~\eqref{eq:viete}, even when the points $(\bx_i, \by_i)$ are
not distinct, but the coordinates $(a_0, a_1, a_2)$ are
not well defined when $\bx_1, \bx_2, \bx_3$ are not
distinct. More precisely, the interpolation $\by_i=\phi_a(\bx_i)$
means that these coordinates are defined by Cramer's rule:
\begin{equation}\label{eq:cramer}
	a_0 = \frac{
		\begin{vmatrix}
			\by_1 & \bx_1 & \bx_1^2\\
			\by_2 & \bx_2 & \bx_2^2\\
			\by_3 & \bx_3 & \bx_3^2
	\end{vmatrix}}
	{\begin{vmatrix}
			1 & \bx_1 & \bx_1^2\\
			1 & \bx_2 & \bx_2^2\\
			1 & \bx_3 & \bx_3^2
	\end{vmatrix}},\quad
	a_1 = \frac{
		\begin{vmatrix}
			1 & \by_1 & \bx_1^2\\
			1 & \by_2 & \bx_2^2\\
			1 & \by_3 & \bx_3^2
	\end{vmatrix}}
	{\begin{vmatrix}
			1 & \bx_1 & \bx_1^2\\
			1 & \bx_2 & \bx_2^2\\
			1 & \bx_3 & \bx_3^2
	\end{vmatrix}},\quad
	a_2 = \frac{
		\begin{vmatrix}
			1 & \bx_1 & \by_1\\
			1 & \bx_2 & \by_2\\
			1 & \bx_3 & \by_3
	\end{vmatrix}}
	{\begin{vmatrix}
			1 & \bx_1 & \bx_1^2\\
			1 & \bx_2 & \bx_2^2\\
			1 & \bx_3 & \bx_3^2
	\end{vmatrix}}.
\end{equation}

We get a very similar picture on the chart $\cU_{(3)}$, simply by
exchanging the roles of the variables $\bx$ and $\by$.  The coordinate
ring of the third chart $\cU_{(2, 1)}$ is slightly different. Its
coordinate ring is $\IC[a_1, a_2, b_1, b_2, c_1, c_2]$, and an ideal
$I_{(a, b, c)}\in \cU_{(2, 1)}$ with coordinates $(a, b, c) = (a_1,
a_2, b_1, b_2, c_1, c_2)$ is given by
\[
I_{(a, b, c)}\coloneqq \latt{ \bx^2-a_0 - a_1\bx-a_2\by, \bx\by-b_0-b_1\bx-b_2\by,\by^2-c_0 - c_1\bx-c_2\by}.
\]
with the following formulae, patiently deduced from~\cite[(2.16)~and~(2.17)]{Haiman}:
\[
	a_0 = a_2(b_1-c_2)+b_2(b_2-a_1),\quad
	b_0 = a_2c_1-b_1b_2,\quad
	c_0 = c_1(b_2-a_1)+b_1(b_1-c_2).
\]

\textit{Restriction to the fibre over the origin}.  Let us restrict
the computation to the variety $\Kum 2 A$. On the chart $\cU_{(1, 1,
  1)}$ the condition $\bx_1+\bx_2 + \bx_3 = 0$ gives $e_1 =0$, and the
condition $\by_1+\by_2 + \by_3 =0$ gives $3 a_0 + a_2 (\bx_1^2 +
\bx_2^2 + \bx_3^2) = 0$, so $3 a_0 - 2 a_2 e_2 = 0$: the local
equations of $\Kum 2 A$ are thus $e_1= 0$ and $a_0 = \frac{2}{3}
a_2e_2$. On the variety~$\Tilde A^3_0$, we have the relations
\begin{equation}\label{eq:model}
	\begin{aligned}
	\bx_3 &= - \bx_1 - \bx_2,&\quad
	\by_1 &= (\bw_1 + \bz_3) \bx_1,&\quad
	\by_2 &= (\bw_2 + \bz_3) \bx_2,\\
	\bx_2 \bw_2 &= -\bx_1 \bw_1,&\quad
	\by_3 &= \bz_3 \bx_3 = - (\bx_1 + \bx_2)\bz_3,
	\end{aligned}
\end{equation}
and we interpret the coordinates $a_i$ as rational maps $\Tilde
 a_i\colon \Tilde A^3_0 \dashrightarrow \IC$. An elementary
computation, starting from Formula~\eqref{eq:cramer} gives:
\begin{align}\label{eq:functions_a}
	\Tilde a_1(\bx_1, \bw_1, \bx_2, \bw_2, \bz_3)& = \bz_3 + \bw_1\bw_2\frac{\bw_1^2+\bw_2^2-4\bw_1\bw_2}{(\bw_1-2\bw_2)(2\bw_1-\bw_2)(\bw_1+\bw_2)},\\
	\Tilde a_2(\bx_1, \bw_1, \bx_2, \bw_2, \bz_3)&=\frac{-3\bw_1\bw_2^2(\bw_1-\bw_2)}{\bx_1(\bw_1-2\bw_2)(2\bw_1-\bw_2)(\bw_1+\bw_2)}.
\end{align}
We see that $\Tilde a_1$ defines a rational function on $\Tilde F_1$,
but $\Tilde a_2$ does not, because of its pole along $\bx_1 = 0$.  Let
$G\subset \Tilde F_1$ be the support of the $1$-cycle defined by the
numerator of this function: that is,
\begin{align}\label{eq:closure_graph_a2}
	G\coloneqq \cV\left(\bw_1\bw_2^2(\bw_1-\bw_2)\right)_{\red}\subset\Tilde F_1.
\end{align}
The fact that the coordinate function $\Tilde a_2$ cannot be extended
to $\Tilde F_1\setminus G$ means that the rational image of $F_1$ by
$\gamma_A$ does not land in the open subset $\cU_{(1, 1, 1)}\cap \Kum
2 A$. By exchanging the roles of the variables $\bx$ and $\by$, we get
that it does not land in the open subset $\cU_{(3)}\cap \Kum 2 A$
either. Since $\Hilb 3 A\smallsetminus\left(\cU_{(1, 1, 1)} \cup
\cU_{(3)}\right) = \left\{Z_\infty\right\}$, the conclusion is that
$\gamma_A$ contracts the generic point of $F_1$ to $Z_\infty$, so the
restriction $\left.\gamma_A\right|_{F_1}$ extends to the whole of
$F_1$ and contracts it to the noncurvilinear point (however,
$\gamma_A$ itself is not defined on the whole $F_1$). This concludes
the proof.
\end{proof}

In Appendix~\ref{app:complements_prop_birational_F1} we give three
alternative arguments, the first one using saturation of ideals, the
second one using the computation on the chart $\cU_{(2, 1)}$ to see
more explicitly the contraction of the divisor $F_1$ and the third one
using explicit projective coordinates.

We now analyse the behaviour of $\gamma_A$ around the divisor $F_2$.
The fixed locus of the natural involution $\Kum 2 \tau$ induced by
$\tau$ on $\Kum 2 A$ consists of the Kummer surface $\Kum 1 A$,
embedded in $\Kum 2 A$ as the locus of subschemes supported on
$0$-cycles of the form $a + (-a) + O_A\in \Symo 3 A$, plus $36$
isolated points (see for instance~\cite[\S4.3.1]{BNWT}). This
embedding of the Kummer surface yields the following diagram:
\[
\xymatrix{\Kum 1 A \ar@{^(->}[r]\ar[d]_\veps & \Kum 2 A\ar[d]^{h_A} \\
	A/\tau \ar@{^(->}[r] & \Symo 3 A
}
\]

\begin{proposition}\label{prop:birational_F2}
The birational map $\gamma_A$ contracts the divisor $F_2$ to the
Kummer surface $\Kum 1 A$ embedded in $\Kum 2 A$.
\end{proposition}

\begin{proof}
Under the embedding of $\Kum 1 A$ in $\Kum 2 A$, if $a\in A$ is a
nonzero $2$-torsion point with image $\bar a\in \Kum 1 A$, the
exceptional fibre $\veps^{-1}(\bar a)$ is sent to the curve
$h_A^{-1}(2a + O_A)$ parametrising the nonreduced length two
subschemes of $A$ supported at~$a$ (the third support point being the
origin). This is a rational curve as it is isomorphic to the
Brian\c{c}on variety $B_a^2\cong\IP^1$. The embedding of the fibre
$\veps^{-1}(\overline{O_A})$ in $B^3_{O_A}$ is similar, and can be
computed as follows, using the same method as in the proof of
Proposition~\ref{prop:birational_F1}. We compute, locally over the
origin, the image of the composite map on the first row:
\[
\xymatrix{\Tilde A \ar[r]^-{2:1} \ar[d]_{\bl_A}&\Kum 1 A
  \ar@{^(->}[r]\ar[d]_\veps & \Kum 2 A\ar[d]^{h_A}
  \\ A\ar[r]^-{2:1} & A/\tau \ar@{^(->}[r] & \Symo 3 A }
\]

Let $\bx, \by$ be local coordinates around the origin of $A$ and
$(\bx, \bz)$ be local coordinates around $\bl_A^{-1}(O_A)$, with
$\by=\bx\bz$. Using notation as above, we put $a = (\bx_1, \by_1) =
(\bx, \by)$, then $-a = (\bx_2, \by_2) = (-\bx, -\by)$ and $(\bx_3,
\by_3) = (0, 0)$. Thus, on the chart $\cU_{(1, 1, 1)}$, the coordinate
functions are
\[
e_1 = 0, \quad e_2= \bx^2, \quad e_3 = 0, \quad a_0 = 0, \quad a_1 =
\bz, \quad a_2 = 0.
\]
Putting $\bx = 0$, we see that the image of
$\veps^{-1}(\overline{O_A})$ in $\Kum 2 A$ consists of the rational
curve in $B_{O_A}^3$ parametrising the subschemes with zero curvature,
\ie where $\beta=0$ in the description given above.

The minimal resolution morphism $\veps$ is the blowup of the classes
of the sixteen $2$-torsion points of $A$.  Since the morphism $\bl_A$
blows up the class of the origin, it factorises through the blowup
$\veps'$ of the classes of the fifteen nonzero ones:
\begin{align}\label{diag:Kummer_surface}
	\xymatrix{\Kum 1 A  \ar[dr]^\veps \ar[d]_{\veps'}&\\
		 \Tilde{A}/{\tilde\tau} \ar[r]^{\bl_A}& A/\tau}
\end{align}
To show that the map $\gamma_A$ contracts the divisor $F_2$ to the
Kummer surface, we first observe that this map is dominated by the
embedding of $\Kum 1 A$ in $\Kum 2 A$, as shown in the following
diagram, where $\pr_1$ denotes the projection to the first factor:
\[
	\xymatrix{ \Kum 1 A\times E_A
          \ar[r]^-{\pr_1}\ar[d]_-{\tilde\veps\times\id}&
          \Kum 1 A \ar@{^(->}[r] & \Kum 2 A
          \ar[d]^{h_A}\\ \Tilde{A}/{\tilde\tau} \times E_A \ar[r]&
          F_2 \ar@{-->}[ur]^{\gamma_A|_{F_2}} \ar[r]_{\Sym 3 {\bl_A}}
          & \Symo 3 A}
\]
But this is not enough to show that the map $\gamma_A$ contracts the
divisor $F_2$, so we use a similar computation as in the proof of
Proposition~\ref{prop:birational_F1}. That is, we compute the rational
map $g$ in a neighbourhood of the premiage $\Tilde{F}_2$ of $F_2$ in
$\Tilde{A}^3_0$, which is given as the image of the morphism:
\[
\Tilde A\times E_1\to \Tilde{A}^3_0, \quad (\tilde a, e)\mapsto (\tilde a, \tilde\tau(\tilde a), e).
\]

Taking local coordinates $(\bx_1, \by_1), (\bx_2, \by_2), (\bx_3,
\by_3)$ of $A^3$ and $\by_i = \bx_1\bz_i$, the divisor~$F_2$ is given
by the relations $\bx_1 + \bx_2 = 0$ and $\by_1 + \by_2=0$, so in the
local coordinates $\bx_1, \bx_2, \bw_1, \bw_2, \bz_3$ of $\Tilde
A^3_0$, it has local equations $\bx_1 + \bx_2 = 0$ and $\bw_1-\bw_2=
0$. In the chart $\cU_{(1, 1, 1)}$, we observed that the local
coordinates of the variety $\Kum 2 A$ in $\Hilb 3 A$ are $(e_2, e_3,
a_1, a_2)$. The coordinate $e_3=0$ is zero along $\Tilde F_2$ and we
see with Equation~\eqref{eq:functions_a} that the function $\Tilde
{a_2}$ extends generically to zero along $\Tilde F_2$. This shows that
$\gamma_A$ contracts $F_2$ to the surface of local equations $e_1 =a_2
= 0$: these are the local equations of $\Kum 1 A$ that we computed
above.
\end{proof}

\section{The genus two curves}\label{s:genustwo}

Let $C$ be a smooth projective curve of genus $g$. For any
$n\geq 0$, we denote by $\Sym n C$ the symmetric product of $C$, by
$\Phi_n\colon C^n \to \Sym n C$ the Chow quotient, by $\Pic n C$ the
moduli space of isomorphism classes of degree $n$ line bundles on $C$
and by $\Hilb n C$ the Hilbert scheme parametrising length $n$
zero-dimensional subschemes of $C$.

For $n\geq 1$ and for any $(p_1, \dots, p_n)\in C^n$, we denote by
$p_1+\cdots+p_n\in \Sym n C$ the formal sum, which we interpret,
depending on the context, as an element of the Chow quotient, as a
divisor on $C$ or as a length $n$ subscheme of $C$, since the
Hilbert--Chow morphism $\Hilb n C\to \Sym n C$ is an isomorphism. We
denote by $\Delta_n\subset \Sym n C$ the locus of nonreduced
subschemes.

For any $p_1+\cdots+p_n\in \Sym n C$, we denote by
$\cO_C(p_1+\cdots+p_n)\in \Pic n C$ the corresponding isomorphism
class of line bundles, or equivalently linear equivalence class
of divisors on $C$. We denote by $\sim$ the linear equivalence
relation between divisors on $C$. We define the Jacobian $\Jac C$ of
$C$ as the group $\Pic 0 C$ of isomorphism classes of degree zero line
bundles on $C$.

From now on, we assume that $C$ is a genus two curve and we put $A =
\Jac C$. All notation introduced above and indexed by an abelian
surface $A$ will be indexed by $C$ for more readibility. That is:
\[
\greg_C\coloneqq \greg_{\Jac C}, \quad h_C\coloneqq h_{\Jac C}, \quad
\gamma_C\coloneqq \gamma_{\Jac C}, \quad E_C\coloneqq E_{\Jac C},
\quad\bar\alpha_C = \bar\alpha_{\Jac C}.
\]

\subsection{The Jacobian of a genus two curve} \label{ss:jacobian}

Let $C$ be a genus two curve.  It is
hyperelliptic: we denote the hyperelliptic involution by $\sigma$ and
the associate ramified canonical double covering by $\pi\colon C\to\IP^1$. The
ramification locus consists of six distinct points, the
\emph{Weierstrass points} of $C$. The curve $C$ thus admits an
equation of the form $z^2=f(x,y)$, where $f$ is a degree six
homogeneous polynomial vanishing at the six branch points, so we may
consider $C$ as a curve in the weighted projective plane
$\IP_{113}$ with homogeneous coordinates $[x:y:z]$. We denote this
embedding of~$C$ by $\iota\colon C\inj \IP_{113}$.

With respect to these coordinates, for $[x:y:z]\in C$ we have
$\pi([x:y:z])=[x:y]$ and $\sigma([x:y:z])=[x:y:-z]$.  We may choose
coordinates on $\IP^1$ such that the branch points of $\pi$ are
$[1:0], [0:1], [1:1]$ and three other distinct points
$[\lambda_i:1]$. Then the equation of $C$ in $\IP_{113}$ is
\begin{align}\label{eq:curve}
	z^2 =  xy(x-y)(x-\lambda_1y)(x-\lambda_2y)(x-\lambda_3y)\eqqcolon f(x, y).
\end{align}
We denote by $\infty \coloneqq [1:0:0]\in C$ the ramification point over $[1:0]$.

Every element of $\Jac C$ has a unique representative, called a
\emph{reduced divisor} (see~\cite[Chapter 3, \S2]{Mumford_TataII}), which is one of
\begin{enumerate}
	\item  $\cO_C(p_1+p_2-2\infty)$ with $p_i\in C\setminus\{\infty\}$, $p_1\neq p_2$ and $p_1\neq\sigma(p_2)$;
	\item $\cO_C(2p-2\infty)$ with $p\in C\setminus\{\infty\}$ and $p\neq\sigma(p)$;
	\item $\cO_C(p-\infty)$ with $p\in C\setminus\{\infty\}$;
	\item $\cO_C$.
\end{enumerate}
In particular, $p+\sigma(p)\sim 2\infty$ for all $p\in C$.  It is a
classical result, see for instance~\cite[Lecture~3]{Mumford_curves},
that the Abel--Jacobi map
\begin{align}\label{eq:abel_jacobi2}
\abel_2\colon \Sym 2 C\to \Jac C, \quad p_1 + p_2 \mapsto \cO_C(p_1 + p_2 - 2 \infty)
\end{align}
is the blowup $\bl_C$ of the origin of $\Jac C$, that is $\Tilde A =
\Tilde{\Jac C} \cong \Sym 2 C$.
The exceptional divisor of the blowup
is thus
\begin{equation}\label{eq:aj_exc}
E_C \coloneqq \{p+\sigma(p)\mid p\in C\}\subset \Sym 2 C.
\end{equation}
The sign involution on $\Jac C$ is given by $\tau(L) = L^{-1}$. For
any $p_1, p_2\in C$ we have
\[
p_1+p_2 +\sigma(p_1) + \sigma(p_2) \sim 4\infty,
\]
so $\cO_C(p_1+p_2-2\infty)^{-1} = \cO_C(\sigma(p_1) + \sigma(p_2) -
2\infty)\in \Jac C$. It follows that $\tau = \sigma^\ast$ and that the
involution $\tilde\tau$ on $\Sym 2 C = \Tilde{\Jac C}$ introduced
in~\ref{ss:divisorial_chow} is given by $\tilde\tau(p_1+p_2) =
\sigma(p_1)+\sigma(p_2)$.

By the Poincar\'e formula~\cite[11.2.1]{BL}, the theta divisor giving
the principal polarisation of $\Jac C$ is the image of the
Abel--Jacobi map $C\to\Jac C$, that is
$\Theta\coloneqq\{p-\infty \mid p\in C\}$, and $\Theta^2 =2$. If $C$
is generic in its moduli space, then $\NS {\Jac C}= \IZ
\Theta$. Following~\S\ref{ss:intro_genKm}, we denote by $h$ the image
of $\Theta$ in $\Kum 2 {\Jac C}$ and we observed above that the
minimal possible polarisation degree of $\Kum 2 {\Jac C}$ is $q_{\Kum
  2 {\Jac C}}(3h - \delta)= 12$.

\begin{remark}
In this setup, Diagram~\eqref{diag:Kummer_surface} has rich and famous
geometric properties. Recall the situation:
\begin{align*}
	\xymatrix{\Kum 1 {\Jac C} \ar[dr]^\veps
          \ar[d]_{\tilde\veps}&\\ \Sym 2 C/{\tilde\tau}
          \ar[r]^{\abel_2}& \Jac C/\tau}
\end{align*}
The double covering $\pi\colon C\to \IP^1$ induces a $4:1$ covering
\[
\Sym 2 \pi\colon \Sym 2 C \to \Sym 2 {\IP^1}\cong \IP^2,
\]
which factors through $\tilde\tau$ to a double covering $\Sym 2
C/{\tilde\tau}\to \IP^2$ whose branch locus is the union of the six
lines
\[
\{w+x\mid x\in \IP^1\}\in \Sym 2 {\IP^1},
\]
where $w$ is one the six branch points of $\pi$. Each of these curves
defines a line in $\IP^2$ and the branch locus is thus a sextic curve
in $\IP^2$ that consists of six lines tangent to a conic, meeting in
$15$ points that are blown up by $\tilde\veps$ (see for instance
\cite[\S 10.2]{GS}, with a complementary point of view). 
\end{remark}

\subsection{The linear system of cubics}\label{ss:lin_syst_cubics}

The linear system of cubics in $\IP_{113}$ is $5$-dimensional:
\[
\HH^0(\IP_{113}, \cO_{\IP_{113}}(3))=\Span(x^3,x^2y,xy^2,y^3,z).
\]
We put $\cL\coloneqq \iota^\ast \cO_{\IP_{113}}(3)$, so $|\cL|$ is the
complete linear system of cubics on the curve~$C$. We have
$h^0(C,\cL)=5$ by Riemann--Roch (see also the simple part of Mattuck's
argument in the proof below) and so
$\HH^0(C,\cL)=\HH^0(\IP_{113},\cO_{\IP_{113}}(3))$ with $|\cL| \cong \IP^4$.

Any divisor $D\in |\cL|$ has an equation of the form:
\begin{align}\label{eq:general_cubic}
\alpha_0 x^3 + \alpha_1 x^2y + \alpha_2 xy^2 + \alpha_3 y^3 + \alpha_4
z = 0,
\end{align}
with $\alpha_i\in \IC$, so a generic $D$ cuts the curve $C$ in six
distinct points (see~Figure~\ref{fig:cubic_interpolation}). It
follows that $\deg(\cL)=6$, so the line bundle $\cL$ is very
ample~\cite[Corollary~IV.3.2]{Hartshorne}.

\begin{figure}[!h]
	\begin{center}
		\includegraphics[width = 4cm, height = 4cm]{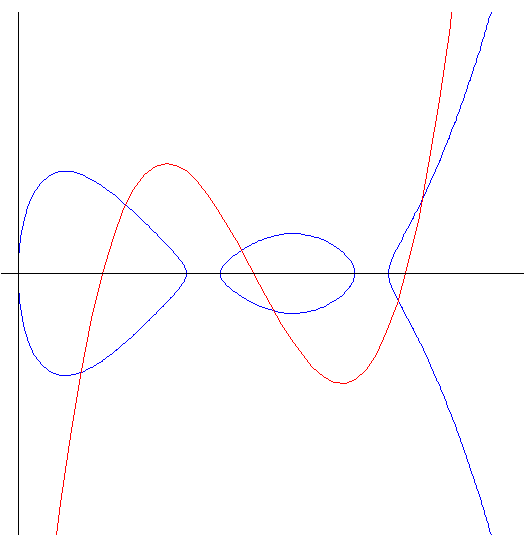}
		\caption{A genus $2$ curve (in blue) and a cubic
                  interpolation (in red) intersecting in $6$ points on
                  the affine chart $y=1$ of coordinates $x$ (abscissa)
                  and $z$ (ordinate).}
		\label{fig:cubic_interpolation}
	\end{center}
\end{figure}

\begin{lemma}\label{lem:6infinity}
We have $|\cL| = |6\infty|$ and $\pi_\ast \cL\cong \cO_{\IP^1} \oplus
\cO_{\IP^1}(3)$.
\end{lemma}

\begin{proof}
The equality $|\cL| = |6\infty|$ is an application to our situation of
classical results of Mattuck and Mumford.  If $C$ is a smooth
curve of genus $g$, and if $n > 2g-2$, then Mattuck~\cite{Mattuck}
shows that the divisor class map
\[
\delta \colon \Sym n C \to \Pic n C, \quad p_1 + \cdots +p_n \mapsto \cO_C(p_1+\cdots+p_n)
\]
is a $\PP^{n-g}$ bundle. For completeness we give the part of
Mattuck's argument that we need: we do not need the local triviality,
which is harder. Writing $\xi\coloneqq p_1 + \cdots + p_n\in
\Sym n C$, the fibre of $\delta$ over $\cO_C(\xi)$ is by definition the
complete linear system $|\xi|$. Since $\deg(K_C - \xi) = 2g-2-n<0$, we
have $h^0(X, K_C-\xi) = 0$, hence $h^0(C, \xi) = n-g+1$ by the
Riemann--Roch theorem. Thus $|\xi|\cong \IP^{n-g}$ and this dimension
does not depend on $\xi$.

We apply this result with $g=2$, for the genus two
curve $C$ introduced above, and we keep the same notation as above.

Secondly, by a result of Mumford~\cite[Chapter~3,\S2]{Mumford_TataII},
if $D$ is a degree $d$ irreducible curve $D$ in $\IP_{113}$ and by
$p_1, \ldots, p_k$ are the not necessarily distinct intersection
points of $D$ with the genus two curve~$C$, then the divisor
$p_1+\cdots + p_k$ is linearly equivalent to $(2d)\infty$. In
particular, we have $k = 2d$.  Again we recall the proof for
completeness. The curve $D$ has equation $g(x, y, z) = 0$, with $g$ a
weighted homogeneous polynomial of degree $d$, and the curve $C$ is given
by the equation~\eqref{eq:curve}.  The quotient $\frac{g(x, y,
  z)}{y^d}$ thus defines a rational function on $C$ whose divisor is:
\[
p_1+\cdots+p_k - (2d) \infty.
\]
Since this divisor has degree zero, we have $k = 2d$.  We apply this
result with $g=2$ and $d=3$ and $k=6$, this gives $|\cL| = |6\infty|$.

The computation of the direct image $\pi_\ast \cL$ is a
straightforward application of the cyclic covering methods~\cite[\S
  I.17]{BHPV}. Since the canonical sheaf of $C$ is $\omega_C\cong
\cO_C(2\infty)$, by the first assertion we get that $\cL\cong
\omega_C^{\otimes 3}$. The double covering $\pi$ is branched at the
six Weierstrass points, so it is determined by the line bundle $L =
\cO_{\IP^1}(3)$ such that $L^{\otimes 2}\cong \cO_{\IP^1}(6)$ and
$\pi_\ast \cO_C \cong \cO_{\IP^1}\oplus L^{-1}$.  We have
$\omega_C\cong \pi^\ast (\omega_{\IP^1}\otimes L) \cong \pi^\ast
\cO_{\IP^1}(1)$, so by the projection formula
\[
\pi_\ast \omega_C^{\otimes 3}\cong \pi_\ast\pi^\ast
\cO_{\IP^1}(3)\cong \cO_{\IP^1}(3)\otimes \pi_\ast \cO_{\IP^1}\cong
\cO_{\IP^1}\oplus \cO_{\IP^1}(3).
\qedhere
\]
\end{proof}

The Abel--Jacobi map is the map
\begin{align}\label{eq:Jacobi6}
	\abel_n\colon \Sym n C \to \Jac C, \quad \quad p_1 + \cdots
        +p_n \mapsto \cO_C(p_1+\cdots+p_n- n\infty),
\end{align}
By Lemma~\ref{lem:6infinity} we have $|\cL|\cong
\abel_6^{-1}(\{\cO_C\})$.  Let us rephrase this: all length $6$
zero-dimensional subschemes of $C$ that admit a cubic interpolation
belong to the fibre over $\cO_C$ of the Abel--Jacobi map $\abel_6$. But since
the linear system on $C$ cut out by the cubics has the correct
dimension, this is a characterisation of this fibre.

\subsection{Interpretation of the linear system of cubics as a degeneracy locus}
\label{ss:degen}

Let us describe the linear system $|\cL|$ differently.  Consider a
length $6$ subscheme $\xi\subset C$, that is $\xi\in \Sym 6 C$. Start
with the exact sequence
\[
0  \To \cI_\xi \To \cO_C \To \cO_\xi \To 0,
\]
and tensor by $\cL$:
\[
0 \To \cL\otimes \cI_\xi\To \cL \To \cL\otimes\cO_\xi\To 0.
\]
Since $\deg(\cL)>\deg(K_C)$, we have by Serre duality $h^1(C,\cL)=0$,
so we get a four terms exact sequence in cohomology:
\begin{equation}\label{4_terms_seq}
	0\rightarrow \HH^0(C,\cL\otimes \cI_\xi)\rightarrow
        \HH^0(C,\cL)\xrightarrow{\rest_\xi}
        \HH^0(C,\cL\otimes\cO_\xi)\rightarrow \HH^1(C,\cL\otimes
        \cI_\xi)\rightarrow 0.
\end{equation}

Since $I_\xi\cong \cO_C(-\xi)$, the sheaf $\cL(-\xi)\coloneqq
\cL\otimes \cI_\xi$ is a degree zero line bundle.  We are interested
in the restriction morphism $\rest_\xi$, whose rank is at most $5$:
\[
\rest_\xi\colon \IC^5\cong \HH^0(C,\cL)\to\HH^0(\xi,\cL\vert_\xi)\cong\IC^6.
\]

\begin{proposition} \label{prop:rank4}
We have $|\cL| = \{\xi\in\Sym 6 C\mid \rank(\rest_\xi) = 4\}$, and
there is no $\xi\in \Sym 6 C$ such that $\rank \rest_\xi < 4$.
\end{proposition}

\begin{proof}
For any $\xi\in \Sym 6 C$, the exact sequence~\eqref{4_terms_seq} says
that
\[
\rank(\rest_\xi) = 5 - h^0(C, \cL(-\xi)),
\]
so $\rank(\rest_\xi) = 4$ means that $h^0(C, \cL(-\xi)) = 1$. Since
$\cL(-\xi)$ has degree zero, this is equivalent to $\cL(-\xi) \cong
\cO_C$, or equivalently to $\xi\in|\cL|$.  Moreover, if $\rank
\rest_\xi < 4$ then $h^0(C, \cL(-\xi)) \geq 2$: this is not possible
since $\cL(-\xi)$ has degree zero.
\end{proof}

Proposition~\ref{prop:rank4} means that for every length~$6$ subscheme
of $C$, there exists at most one cubic interpolation.  In what
follows, we will frequently interpret the identification
$\abel_6^{-1}(\{\cO_C\}) = |\cL|$ as the isomorphism
\begin{align}\label{eq:interpretation}
	\xymatrix@R-2pc{ \abel_6^{-1}\left(\{\cO_C\}\right) \ar[r]&
          |\cL|\cong \IP \HH^0(C, \cL) \\ \xi \ar@{|->}[r] &
          \ker(\rest_\xi) }
\end{align}
that sends a length $6$ subscheme admitting a cubic interpolation to
the equation of this uniquely defined cubic.

\begin{remark}
By Proposition~\ref{prop:rank4}, there exists at most one cubic
interpolation for every given length~$6$ subscheme of $C$. Although
this is expected for general subschemes, it is remarkable that it
holds for all of them. The
basic general observation is that a cubic interpolation can never
factor as conic and a ``nonvertical" line (this expression makes sense in the
affine chart $y=1$ with coordinates $(x, z)$, see
Figure~\ref{fig:cubic_interpolation_vertical}), because of
shape of the equation and the fact that the variable $z$ has degree~$3$. If the cubic
contains a ``vertical line" $x=\alpha$, then its equation does not
contain the variable $z$ and it thus factors as a product of three
vertical lines.
\end{remark}

\subsection{The linear system of conics} \label{ss:systemofconics}

Similarly as in \S\ref{ss:lin_syst_cubics}, the linear system of
conics in $\IP_{113}$ is $3$-dimensional:
\[
\HH^0(\IP_{113}, \cO_{\IP_{113}}(2)) = \Span(x^2, xy, y^2).
\]
We put $\cC\coloneqq\iota^\ast \cO_{\IP_{113}}(2)$.  Let $\xi\in\Sym 4
C$. Similarly as in \S\ref{ss:degen}, to study the cubic
interpolations that pass through four points we make use of the restriction morphism
\[
\rest_\xi \colon \IC^5\cong\HH^0(C, \cL)\to \HH^0(\xi, \left.\cL\right|_\xi)\cong \IC^4.
\]
This time, $\ker(\rest_\xi)$ is never zero so there exists at least
one cubic interpolation at~$\xi$, and it is unique if and only if
$h^0(C, \cL(-\xi)) = 1$.

We now characterise those length four
subschemes of $C$ that admit a conic interpolation, and we relate this
to the extent to which the cubic interpolations passing through these
points fail to be unique.

\begin{lemma} \label{lem:equiv_conic}
Let $\xi\coloneqq p_1+ p_2+ p_3+ p_4\in \Sym 4 C$. The following
assertions are equivalent:
\begin{enumerate}
	\item\label{lem:conic_1} $\xi\sim 4\infty$.
	\item\label{lem:conic_2} There exists $K\in |\cC|$ such that $K\cap C = p_1+\cdots+p_4$.
	\item\label{lem:conic_3} Up to permutation of the points, $p_2 = \sigma(p_1)$ and $p_4 = \sigma(p_3)$.
	\item\label{lem:conic_4} $h^0(C, \cL(-\xi))=2$.
\end{enumerate}
\end{lemma}

\begin{proof}
Since the equation of a conic $K\in |\cC|$ does not contain the
variable $z$, within the affine chart $y=1$ it consists of two
vertical lines (or one double line). Therefore it cuts $C$ in four
points, which
form two orbits under the action of the involution~$\sigma$, so up to
reordering the points, we conclude that $p_2 = \sigma(p_1)$ and $p_4 =
\sigma(p_3)$. Since the converse is clear, this proves
(\ref{lem:conic_2})$\Leftrightarrow$(\ref{lem:conic_3}). Moreover, any
line of equation $ax+by=0$, with $(a:b)\in \IP^1$, completes this
conic to cubic interpolation that consisting of three
vertical lines. This means that there is a pencil of cubics
intersecting $C$ at $\xi$, so $h^0(C, \cL(-\xi))=2$.  This proves
(\ref{lem:conic_2})$\Rightarrow$(\ref{lem:conic_4}). Since $p +
\sigma(p)\sim 2\infty$ for any $p\in C$ (see~\cite[Chapter 3,
  \S2]{Mumford_TataII}), we
get
\[
p_1 + \cdots+p_4 \sim 4\infty.
\]
Conversely, if four points $p_1, \ldots, p_4$ satisfy the relation
$p_1+\cdots+p_4\sim 4\infty$, arguing similarly as in the proof of
Lemma~\ref{lem:6infinity}, we see that $|\cC| = |4\infty|$ and we
deduce that these four points admit a unique conic interpolation.
This proves (\ref{lem:conic_1})$\Leftrightarrow$(\ref{lem:conic_2}).

It remains to show that
(\ref{lem:conic_4})$\Rightarrow$(\ref{lem:conic_2}).
Suppose that (\ref{lem:conic_2}) does not hold, so at least three of the
$x$-coordinates of the points $p_i$ are distinct.

First assume that the $x$-coordinates of the four points are all
different from one another and from $\infty$, so that we may write
them as $p_i=(x_i:1:z_i)$. There exists a unique cubic interpolation
$g(x, z)$ at these points, defined by the four conditions $g(x_i) =
z_i$, so $h^0(C, \cL(-\xi))=1$.  If instead one of the points is
$\infty$, there is still a cubic interpolation, but it has no
$x^3$-term and it is uniquely determined by the interpolation at the
three remaining points.

If only three of the $x$-coordinates are distinct, we may assume that
$p_2 = \sigma(p_1)$ and that the coordinates $x_1, x_3, x_4$ are
distinct. To construct a cubic interpolation at these four points, we
first need to take the line joining $p_1$ and $p_2$. Then the only way
to interpolate~$C$ at $p_3$~and~$p_4$ with a conic is to take the
lines joining $p_3$~to~$\sigma(p_3)$ and $p_4$~to~$\sigma(p_4)$, so
these points admit a unique cubic interpolation and $h^0(C,
\cL(-\xi))=1$.
\end{proof}

\section{A degree 15 covering of the linear system of cubics}\label{s:covering}

\subsection{Covering maps and Galois closure} \label{ss:galoisclosure}

Inequivalent notions of covering map coexist in the literature. We follow
\cite[\S3]{GKP_etale} and make the definition below.  Note that we
do not require a covering map to be \'etale.

\begin{definition}\label{def:covering}
A \emph{covering map} is a finite surjective morphism $f\colon X\to
X'$ between normal projective varieties. A covering map $f$ is called
\emph{Galois} if there exits a finite group $G\subset \Aut(X)$ such
that $f$ is isomorphic to the quotient map $X\to X/G$.
%\end{enumerate}
\end{definition}

Let $f\colon X\to X'$ be a morphism between normal projective
varieties.  The support~$R$ of the sheaf of relative K\"ahler
differentials
\[
\Omega_{X/X'}\coloneqq \coker\left(f^\ast \Omega_{X'} \to \Omega_X\right),
\]
endowed with its structure of a closed subscheme of $X$, is the
\emph{ramification scheme} of~$f$. Its image $B\coloneqq f(R)$,
defined as a closed subscheme of $X'$, is the \emph{branch scheme}
of~$f$.  When $X$ is normal and $X'$ is nonsingular, by the
Zariski--Nagata purity theorem~\cite{Nagata, Zariski}, $R$ and $B$ are
divisors on $X$ and $X'$ respectively (see
also~\cite[Theorem~2.4]{Zong}). The terms \emph{ramification locus}
and \emph{branch locus} refer to the underlying sets of closed points.

\begin{theorem}\cite[Theorem~3.7]{GKP_etale}\label{th:closure}
Let $f\colon X\to X'$ be a covering map between quasi-projective
varieties. There exists a normal, quasi-projective variety $\Hat
X$ and a finite surjective morphism $\hat f\colon \Hat X \to
X$, called the \emph{Galois closure} of $f$, such that
\begin{enumerate}
	\item there exist finite groups $H\subset G$ such that the
          morphisms $F\coloneqq f\circ \hat f$ and~$\hat f$ are
          Galois coverings with respective groups $G$ and $H$.
        \item The branch loci of $F$ and $f$ are equal.
\end{enumerate}
\end{theorem}

\subsection{Organising six points on the curve into three pairs}\label{ss:group_H}

\begin{definition}\label{def:groupH}
We denote by $H$ the subgroup of the symmetric group $\kS_6$
generated by the permutations $(1, 2)$, $(1, 3)(2, 4)$ and $(1, 5)(2,
6)$.
\end{definition}

The group $H$ has order $48$ and hence index $15$ in $\kS_6$, but it is not normal.

\begin{lemma}\label{lem:factorise_chow}
The Chow quotient $\Phi_6\colon C^6\to \Sym 6 C$ factorises through a
finite morphism $\varphi\colon \Sym 3{\Sym 2 C} \to \Sym 6
C$ of degree $15$, as follows:
\[
\Phi_6\colon\xymatrix@R=.5pc{ C^6 \ar[r]^-{\Hat\varphi}& \Sym 3
  {\Sym 2 C}\ar[r]^-{\varphi} & \Sym 6 C\\
  (p_1, \ldots, p_6)\ar@{|->}[r] &
  \left[p_1+p_2\right]+\left[p_3+p_4\right]+\left[p_5+p_6\right]
  \ar@{|->}[r] & p_1+\cdots+p_6.  }
\]
The morphism $\Hat\varphi$ is the Galois closure of $\varphi$.
\end{lemma}

\begin{proof}
The morphism $\varphi$ is quasi-finite, hence finite by Stein
factorisation since all varieties involved are projective. Its degree
is $15=\frac{1}{3!}\binom{6}{2}\binom{4}{2}$ and it is clearly
surjective. It is also flat since $\Sym 6 C$ is nonsingular and $\Sym
3 {\Sym 2 C}$ is Cohen--Macaulay (see~\cite[18.17]{Eisenbud}).  The
Chow quotient $\Phi_6\colon C^6\to \Sym 6 C$ is a Galois covering with
group the symmetric group $\kS_6$.  It is easy to check that
$\Hat\varphi$ is the quotient of $C^6$ by the group $H$ introduced in
Definition~\ref{def:groupH}. This group permutes the three pairs and
the position of the points in each pair. Since $\Sym 3 {\Sym 2 C}$ is
normal, the morphism $\varphi$ is a covering map in the sense of
Definition~\ref{def:covering} but it is not a Galois covering since
$H$ is a nonnormal subgroup of $\kS_6$. The morphism $\Hat\varphi$ is
the Galois closure of $\varphi$ in the sense of
Theorem~\ref{th:closure}.
\end{proof}

The ramification scheme of $\Phi_6$ is the big diagonal $\cD_6\subset
C^6$, \ie the union of the closed subschemes defined by the equalities
$p_i = p_j$ for $p = (p_1, \ldots, p_6)\in C^6$. It is a reduced and
reducible divisor. The branch scheme $\Phi_6(\cD_6)=\Delta_6\subset
\Sym 6 C$ is the locus of nonreduced subschemes: it is a reduced and
irreducible divisor, and it is clearly the branch scheme of
$\varphi$. We study now the ramification scheme of $\varphi$.

\begin{remark}\label{rem:explicit_fibre}
Let us describe the fibres of $\varphi$ explicitly. For any
point $(p_1, \ldots, p_6)\in C^6$, we write for short $p_{i,
  j}\coloneqq p_i + p_j\in\Sym 2 C$. The fibre of~$\varphi$ over a
generic point $p_1+\cdots+p_6\in \Sym 6 C$ is the following set of
$15$~points in $\Sym 3{ \Sym 2 C}$:
\[
	\begin{array}{ccc}
	p_{1, 2} + p_{3, 4} + p_{5, 6} &
	p_{1, 2} + p_{3, 5} + p_{4, 6} &
	p_{1, 2} + p_{3, 6} + p_{4, 5} \\
	p_{1, 3} + p_{2, 4} + p_{5, 6} &
	p_{1, 3} + p_{2, 5} + p_{4, 6} &
	p_{1, 3} + p_{2, 6} + p_{4, 5} \\
	p_{1, 4} + p_{2, 3} + p_{5, 6} &
	p_{1, 4} + p_{2, 5} + p_{3, 6} &
	p_{1, 4} + p_{2, 6} + p_{3, 5} \\
	p_{1, 5} + p_{2, 3} + p_{4, 6} &
	p_{1, 5} + p_{2, 4} + p_{3, 6} &
	p_{1, 5} + p_{2, 6} + p_{3, 4} \\
	p_{1, 6} + p_{2, 3} + p_{4, 5} &
	p_{1, 6} + p_{2, 4} + p_{3, 5} &
	p_{1, 6} + p_{2, 5} + p_{3, 4}
	\end{array}
\]
Over a generic point of $\Delta_6$, say for instance when $p_1 = p_2$,
this fibre contains only nine closed points:
\begin{align}\label{eq:nine}
\begin{array}{ccc}
	p_{1, 1} + p_{3, 4} + p_{5, 6} &
	p_{1, 1} + p_{3, 5} + p_{4, 6} &
	p_{1, 1} + p_{3, 6} + p_{4, 5} \\
	p_{1, 3} + p_{1, 4} + p_{5, 6} &
	p_{1, 3} + p_{1, 5} + p_{4, 6} &
	p_{1, 3} + p_{1, 6} + p_{4, 5} \\
	p_{1, 4} + p_{1, 5} + p_{3, 6} &
        p_{1, 4} + p_{1, 6} + p_{3, 5} &
	p_{1, 5} + p_{1, 6} + p_{3, 4} \\
\end{array}
\end{align}
In the scheme-theoretic fibres, the first three points, on the top
line of the list~\eqref{eq:nine}, are nonreduced subschemes with a
length two subscheme supported at $p_1$.
The remaining six points are
(generically) reduced subschemes obtained when the point~$p_1$ is in
the support of two different summands $p_{1, i}$ and $p_{1, j}$ at the same time.
\end{remark}

From Remark~\ref{rem:explicit_fibre} we deduce that a generic fibre of
$\varphi$ consists of nine distinct points that belong to two distinct
divisors $R_1$ and $R_2$, defined below.  We will see shortly that
only $R_2$ is indeed a ramification divisor.  Consider first the Chow
quotient:
\[
q\colon \Sym 2 C \times \Sym 2 C \times \Sym 2 C \to \Sym 3{\Sym 2 C}
\]
and define
\begin{align}\label{div:R1}
R_1 \coloneqq q\left(\Delta_2 \times \Sym 2 C \times \Sym 2 C \right).
\end{align}
It is an irreducible divisor parametrising unordered triples of
effective degree~$2$ divisors on $C$ such that at least one of them is
nonreduced.  Over a generic point of~$\Delta_6$, three points of the
fibre of $\varphi$ belong to $R_1$ (for instance those on the top line
of~\eqref{eq:nine}).

Let us now introduce the \emph{double incidence} subvariety:
\[
\Xi\coloneqq\{(x, \xi_1, \xi_2)\in C\times \Sym 2 C\times \Sym 2 C\mid x\in \Supp(\xi_1)\cap\Supp(\xi_2)\},
\]
where $\Supp(\xi)$ is the set-theoretic support of the subscheme
$\xi$.  The locus $\Xi$ is a codimension~$2$ irreducible subvariety of
the product variety. Denote the projection by
\[
\pi\colon C\times \Sym 2 C\times \Sym 2 C\to \Sym 2 C\times \Sym 2
C,\quad (x, \xi_1, \xi_2)\mapsto (\xi_1, \xi_2)
\]
and define
\begin{align}\label{div:R2}
R_2\coloneqq q\left(\pi(\Xi)\times \Sym 2 C\right).
\end{align}
Over a generic point of $\Delta_6$, six points of the fibre of
$\varphi$ belong to $R_2$ (for instance those not on the top line
of~\eqref{eq:nine}).

\begin{lemma}\label{lem:ram_phi}
We have $\varphi^\ast \Delta_6 = R_1 + 2 R_2$.  The ramification
scheme of $\varphi$ is the reduced and irreducible divisor~$R_2$.
\end{lemma}

\begin{proof}
The ramification divisor of $\varphi$ decomposes as a sum of
irreducible components with multiplicity as $r_1 R_1 + r_2
R_2$. Denote by $e_i$, $i=1, 2$ the local degrees (or branching
orders) of $\varphi$ at generic points of $R_i$, so that $\varphi^\ast
\Delta_6 = e_1 R_1 + e_2 R_2$. By symmetry, these degrees do not
depend on the choice of one of the three (for $i=1$), respectively six
(for $i=2$) preimage points in~$R_i$.  Since $\varphi$ is generically
$15:1$, we have $3 e_1 + 6 e_2 = 15$, and this gives two
possibilities: either $(e_1, e_2) = (1, 2)$ or $(e_1, e_2) = (3, 1)$.

Let us exclude the second possibility. We use the notation of
Remark~\ref{rem:explicit_fibre}. It is enough to consider one point in
$R_2$, say for instance $p_{1, 3} + p_{1, 4} + p_{5, 6}$. This point
is obtained as the limit point of the two reduced subschemes $p_{1, 3}
+ p_{2, 4} + p_{5, 6}$ or $p_{2, 3} + p_{1, 4} + p_{5, 6}$ when $p_2$
goes to $p_1$, so the local degree is two. In comparison, a point
$p_{1, 1} + p_{3, 4} + p_{5, 6}$ in $R_1$ is the limit of $p_{1, 2} +
p_{3, 4} + p_{5, 6}$ when $p_2$ goes to~$p_1$, but there is only one
limit direction since $C$ is a smooth curve, so there is no
ramification there. It follows that $e_1 = 1$ and $e_2 = 2$. We
conclude using~\cite[Lemma~I.16.1]{BHPV} that the ramification divisor
is~$R_2$.
\end{proof}

Let us rephrase the upshot of the above proof: the preimage by
$\varphi$ of the branch divisor $\Delta_6$ is the union of the
divisors $R_1$ and $R_2$, but the preimage points in $R_1$ are not
ramification points, whereas the premiage points in $R_2$ are
ramification points with branching order $2$. The ramification divisor
is thus the reduced divisor $R_2$.

\subsection{The degree 15 covering}\label{ss:def_G}

Recall that we denote by $\greg_C\coloneqq \greg_{\Jac C}$ the variety
$\Symo 3 {\Sym 2 C}$ introduced in Definition~\ref{def:greg}, using
the identification $\Tilde{\Jac C} \cong \Sym 2 C$ explained
in~\eqref{eq:abel_jacobi2}.

\begin{proposition}\label{prop:multiple_4plane}
The morphism $\varphi\colon \Sym 3 {\Sym 2 C} \to \Sym 6 C$ restricts
to a degree $15$ finite morphism $\psi\colon \greg_C\to |\cL|\cong
\IP^4$. The variety $\greg_C$ is normal, geometrically
Cohen--Macaulay, $\IQ$-factorial and Gorenstein, with quotient
singularities.
\end{proposition}

\begin{proof}
The isomorphism \eqref{eq:interpretation} can be formulated
equivalently as a closed embedding $|\cL|\inj \Sym 6 C$ sending a
cubic $D$ to the subscheme $D\cap C$ considered as a formal sum of
points, where the multiplicity of $D\cap C$ at a point $p$ is the
length of the artinian ring $\cO_{p, D\cap C}$. It is easy to check
that the morphism $\abel_6\circ \varphi$ factorises by $\Sym 3
{\bl_C}$, that is:
\[
	\xymatrix{ \Sym 3 {\Sym 2 C}\ar[r]^-{\varphi}\ar[d]_{\Sym 3
            {\bl_C}} & \Sym 6 C\ar[r]^{\abel_6} & \Jac C\\ \Sym 3
          {\Jac C}\ar[rru]_{\bar\alpha_C} }
\]

Since $|\cL| = |6\infty| \cong \abel_6^{-1}({\cO_C})$, this
shows that $\greg_C$ is the fibre of $\varphi$ over $|\cL|$. We denote
by $\psi$ the restriction of $\varphi$ to $\greg_C$.  Since a generic
cubic interpolation cuts $C$ in $6$ points, the
generic fibre of $\psi$ is reduced, so the morphism
$\psi\colon \greg_C\to |\cL| = \IP^4$ is finite of degree $15$.

The surface $\Sym 2 C$ is nonsingular, so the symmetric quotient $\Sym
3 {\Sym 2 C}$ has rational singularities. Since the group $\kS_3$ acts
on it without quasi-reflections, the quotient is Gorenstein (see for
instance~\cite{Ito} and references therein).
We know by
Proposition~\ref{prop:GA_normal} that $\greg_C$ is normal,
Cohen--Macaulay, $\IQ$-factorial with quotient singularities. Since
$\Sym 6 C$ is nonsingular, the finite morphism $\varphi$ is
Gorenstein, meaning that its relative dualising sheaf is locally
free. Since its formation commutes with base change, the fibre
$\greg_C$ over $|\cL|$ is Gorenstein (see for
instance~\cite[\S1]{CE}).
\end{proof}

\begin{definition}\label{def:marc}
We denote by $\marc_C$ the scheme-theoretic fibre over the origin of
the morphism:
\[
C^6\xrightarrow{\Phi_6} \Sym 6 C \xrightarrow{\abel_6} \Jac C,
\]
that is, $\marc_C\coloneqq \left(\abel_6\circ
\Phi_6\right)^{-1}\left(\{\cO_C\}\right)$.
\end{definition}

Let us summarise the situation in the following diagram, where, as we
proved above, we have $\greg_C=
\left(\abel_6\circ\varphi\right)^{-1}\left(\{\cO_C\}\right)$:
\begin{align}\label{diag:galois_closures}
\xymatrix{\marc_C\ar@{^(->}[r]\ar[d]_{\Hat\psi}\ar@/_2pc/[dd]_\Psi & C^6\ar[d]^{\Hat\varphi}\ar@/^4pc/[dd]^{\Phi_6} \\
	\greg_C \ar@{^(->}[r] \ar[d]_\psi& \Sym 3{\Sym 2 C}\ar[d]^\varphi \\ \IP^4 = |\cL|\ar@{^(->}[r]\ar[d] & \Sym 6 C\ar[d]^{\abel_6}\\
	\{\cO_C\} \ar@{^(->}[r] & \Jac C}
\end{align}

\begin{proposition}\label{prop:marc}\text{}
\begin{enumerate}
	\item\label{prop:marc_item1} The scheme $\marc_C$ is a local
          complete intersection. It is a normal variety of
          dimension~$4$, Cohen--Macaulay and Gorenstein.
	\item\label{prop:marc_item2} The morphism
          $\Hat\psi\colon \marc_C\to \greg_C$ is the Galois
          closure of $\psi$ and $\Hat\psi$ is the quotient map
          by the group $H$, that is $\greg_C = \marc_C/H$.
	\item\label{prop:marc_item3} The morphism $\Psi\colon
          \marc_C\to |\cL|$ is syntomic.
\end{enumerate}
\end{proposition}

We recall that  \emph{syntomic} means a flat local complete intersection morphism of locally finite presentation, see~ \cite[\href{https://stacks.math.columbia.edu/tag/01UB}{Definition~29.30.1}]{stacks-project}.

\begin{proof}
For any $p\coloneqq (p_1, \ldots, p_6)\in C^6$, we consider the
restriction morphism $\rest_{\Phi_6(p)}$ introduced above:
\[
\rest_{\Phi_6(p)}\colon \IC^5\cong \HH^0(C, \cL)\to \HH^0(\Phi_6(p), \left.\cL\right|_{\Phi_6(p)})\cong \IC^6.
\]
By definition, the closed points of $\marc_C$ are those sextuples of
points on $C$ that are interpolated (with multiplicity if necessary)
by a cubic. The finite surjective morphism $\Psi\colon \marc_C\to
|\cL|$ maps any $p\in \marc_C(\IC)$ to $\ker(\rest_{\Phi_6(p)})$,
using the isomorphism~\eqref{eq:interpretation}. It follows that
$\marc_C$ is equidimensional of dimension $4$. By
Proposition~\ref{prop:rank4}, the cubic interpolation is always unique
whenever it exists since the locus of points $p\in C^6$ such that
$\rank(\rest_{\phi_6(p)})<4$ is empty, so we have:
\begin{align*}
	\marc_C&= \{p\in C^6\mid \exists s\in \HH^0(C, \cL),\, \cZ(s)
        = \Phi_6(p)\}\\ &= \{p\in C^6\mid \rank(\rest_{\Phi_6(p)})
        \leq 4\},
\end{align*}
where $\cZ(s)$ means the zero scheme of $s$ (we refer to~\cite{AK} for
the definition of this scheme structure).

The local equations of $\marc_C$ at a point $p$ are thus the six
$5\times 5$ minors of any $6\times 5$ matrix $R$ associated to
$\rest_{\Phi_6(p)}$. At least one $4\times 4$ minor, say the
determinant $|R'|$ of the $4\times 4$ submatrix $R'$, is
nonzero at $p$ since the matrix has rank $4$.

The determinants of the two $5\times 5$ submatrices containing~$R'$
vanish at~$p$ and by basic linear algebra, this forces the vanishing
at~$p$ of all the $5\times 5$ minors. So $\marc_C$ is locally given by
two equations in the nonsingular variety $C^6$, hence it is a local
complete intersection. It thus satisfies Serre's condition~$\SerreS k$ for
any~$k\geq 1$.

Using these local equations, we show that $\marc_C$ is regular in
codimension zero (Serre's condition $\SerreR 0$). Since a generic
cubic cuts $C$ in six points whose images under the double
covering $\pi\colon C \to \IP^1$ are distinct,
each
irreducible component of $\marc_C$ contains a dense open subset of
points $p=(p_1, \ldots, p_6)\in \marc_C$ such that the $x$-coordinates
of all $p_i$ are distinct.  Let us show that $\marc_C$ is nonsingular
at these points. Without loss of generality, we may assume that none
of these points is $\infty$, so we denote their coordinates by $p_i =
[x_i:1:z_i]$ with $x_i$ distinct. The matrix of $\rest_{\Phi_6(p)}$ is
the $6\times 5$ matrix
\[
R:=	\begin{pmatrix}
		x_1^3 & x_1^2  & x_1 & 1 & z_1\\
		\vdots & \vdots  & \vdots & \vdots & \vdots \\
		x_6^3 & x_6^2  & x_6  & 1 & z_6
	\end{pmatrix}.
\]
We may take $R'$ to be the top left $4\times 4$ submatrix, so that $|R'|$ is the
Vandermonde determinant $V(x_1,x_2,x_3,x_4)$, which is nonzero because
the $x_i$ are distinct. Hence, as discussed above, the two local
equations of $\marc_C$ at $p$ are
\[
M_5\coloneqq\left|
	\begin{matrix}
		x_1^3 & x_1^2  & x_1 & 1 & z_1\\
		\vdots & \vdots  & \vdots & \vdots & \vdots \\
		x_4^3 & x_4^2  & x_4  & 1 & z_4\\
		x_6^3 & x_6^2  & x_6  & 1 & z_6
	\end{matrix}\right| = 0
\text{ and }
M_6\coloneqq\left|
        \begin{matrix}
	x_1^3 & x_1^2  & x_1 & 1 & z_1\\
	\vdots & \vdots  & \vdots & \vdots & \vdots \\
	x_4^3 & x_4^2  & x_4  & 1 & z_4\\
	x_5^3 & x_5^2  & x_5  & 1 & z_5
        \end{matrix}\right| = 0.
\]
These satisfy
\[
\frac{\partial M_5}{\partial z_5} = \frac{\partial
  M_6}{\partial z_6} = 0, \qquad \frac{\partial M_5}{\partial z_6}=
\frac{\partial M_6}{\partial z_5}  = |R'|\neq 0,
\]
so the Jacobian matrix of $(M_5, M_6)$ at $p$ has rank $2$. This
shows that $\marc_C$ is nonsingular at $p$.  Since $\marc_C$ is $\SerreR 0$
and $\SerreS 1$, it is reduced.

The same argument as in Lemma~\ref{lem:ram_phi} (the computation of
the local degrees) shows that the ramification scheme $R_2\cap
\marc_C$ of $\psi$ is reduced, so the branch scheme $B\coloneqq
\psi(R_2\cap \marc_C)$ is reduced (see
\cite[\href{https://stacks.math.columbia.edu/tag/056B}{Lemma~29.6.7}]{stacks-project}). Locally
over $\IP^4$, the variety~$\marc_C$ is given by a polynomial equation
of the form
\[
	P(x, y_1, \ldots, y_4) = x^{15} + \sum_{i=0}^{14} a_i(y_1,
        \ldots, y_4) x^i = 0,
\]
where $a_i$ are regular functions on affine charts of $\PP^4$ with
coordinates $(y_1, \ldots, y_4)$. The branch scheme $B$ is
the vanishing locus of the discriminant $D(y_1, \ldots, y_4)$ of the
polynomial $P$. Since $B$ is reduced, its singularities are given by
the equations
\[
	D(y_1, \ldots, y_4) = 0, \quad \frac{\partial D}{\partial
          y_i}(y_1, \ldots, y_4) = 0, \quad\forall i=1, \ldots, 4.
\]
A point with local coordinates $(x, y_1, \ldots, y_4)$ is a singular
point of $\marc_C$ if
\begin{align*}
	P(x, y_1, \ldots, y_4) = 0, \quad \frac{\partial P}{\partial
          x}(x, y_1, \ldots, y_4) = 0 \\ \text{ and } \frac{\partial
          P}{\partial y_i}(x, y_1, \ldots, y_4) = 0, \quad\forall i=1,
        \ldots, 4.
\end{align*}
An explicit computation shows that these conditions imply that $(y_1,
\ldots, y_4)$ is a singular point of $B$,
see~\cite[Theorem~4.2]{Sharipov}. This means that $\marc_C$ is
regular in codimension one (condition~$\SerreR 1$). Since it is $\SerreS 2$, it is
normal by Serre's criterion~\cite[Proposition~II.8.23]{Hartshorne}.

Since $\marc_C$ is a local intersection scheme, it admits a Koszul
complex providing a locally free finite resolution
$K_\bullet\to\cO_{\marc_C}$ of its structure sheaf considered as a
$\cO_{C^6}$-module. A standard argument (see for instance~\cite{AK})
produces a spectral sequence
\[
	E_1^{p, q} = \HH^q(C^6, K_p)\Rightarrow H^{p+q}(\marc_C, \cO_{\marc_C}),
\]
from which we deduce that $\HH^0(\marc_C, \cO_{\marc_C})$ is a
quotient of $\HH^0(C^6, \cO_{C^6})$, which is 1-dimensional since
$C$ is connected. It follows that $\marc_C$ is connected.  Since
$\marc_C$ is normal and connected, using Zariski's Main Theorem we
deduce that it is irreducible. Finally, $\marc_C/\kS_6 \cong
\IP^4$. As $\marc_C$ is a local complete intersection scheme, it is
Cohen--Macaulay and Gorenstein~\cite[Corollary 21.19]{Eisenbud}. This
proves assertion~(\ref{prop:marc_item1}).

The variety $\greg_C$ is normal by
Proposition~\ref{prop:multiple_4plane} and it follows from
Lemma~\ref{lem:factorise_chow} and
Diagram~\eqref{diag:galois_closures} that the morphism
$\hat\psi\colon\marc_C\to\greg_C$ is the Galois closure of $\psi$,
that is, $\greg_C = \marc_C/H$; this proves
assertion~(\ref{prop:marc_item2}).  The morphism
$\Psi\colon\marc_C\to\IP^4$ is a finite, hence flat, morphism from a
normal local complete intersection variety to a regular variety,
by ``magic flatness'' (see~\cite[18.17]{Eisenbud}). It follows that
$\Psi$ is also a local complete intersection morphism, hence syntomic
(see \cite[\href{https://stacks.math.columbia.edu/tag/068E}{Lemma
    37.62.8 and Lemma 37.62.12}]{stacks-project}); this proves
assertion~(\ref{prop:marc_item3}).
\end{proof}

\subsection{The rational contraction, revisited} \label{ss:revisit}

We know from Propositions~\ref{prop:birational_F1} and
\ref{prop:birational_F2} that the birational map $\gamma_C\colon
\greg_C\dashrightarrow \Kum 2 {\Jac C}$ contracts the divisor $F_1$
defined in Equation~\eqref{def:F1} to the noncurvilinear point $3
\cO_{C}$ and contracts $F_2$ (see Equation~\eqref{def:F2}) to the
Kummer surface $\Kum 1 {\Jac C}$ naturally embedded in $\Kum 2 {\Jac
  C}$. In this setup, these two divisors have very nice and concrete
descriptions since they parametrise the possible configurations of
triples of pairs of points that are interpolated by cubics consisting
of three ``vertical'' lines (see
Figure~\ref{fig:cubic_interpolation_vertical}).
\begin{enumerate}
	\item \textit{The comb}. Recall that the exceptional divisor
          $E_C\subset \Sym 2 C$ consists of the $0$-cycles of the form
          $p+\sigma(p)$ for $p\in C$, so
\[
	F_1=\{[p_1+\sigma(p_1))+ [p_2+\sigma(p_2)] + [p_3 + \sigma(p_3)]\mid p_1, p_2, p_3\in C\}.
\]
	\item \textit{The cross}. Similarly
\[
	F_2= \{[p_1+\sigma(p_2)] + [p_2+\sigma(p_1)] + [p_3 + \sigma(p_3))]\mid p_1, p_2, p_3\in C\}.
\]
Given $p_1, p_2\in C$, consider the curve:
\[
	\ell_{p_1, p_2}\coloneqq\{[p_1+\sigma(p_2)] + [p_2+\sigma(p_1)] + [p_3 + \sigma(p_3)]\mid p_3\in C\}.
\]
Clearly $\ell_{p_1, p_2}\cong E_C$ is a rational curve and the divisor
$F_2$ is ruled by these rational curves.
\end{enumerate}

\begin{figure}[!h]
	\begin{center}
		\includegraphics[width = 6cm, height = 6cm]{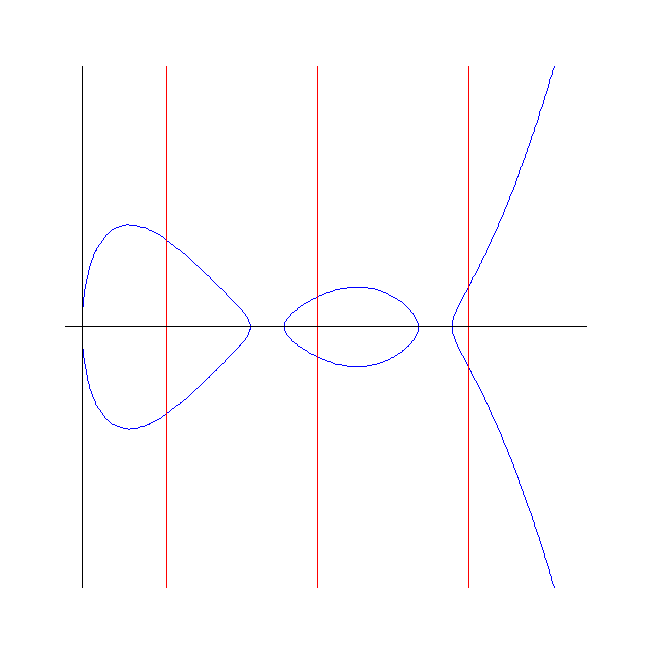}
		\caption{A genus $2$ curve (in blue) and a cubic
                  interpolation (in red) intersecting consisting in
                  three vertical lines intersecting in $6$ points on
                  the affine chart $y=1$ of coordinates $x$ (abscissa)
                  and $z$ (ordinate).}
		\label{fig:cubic_interpolation_vertical}
	\end{center}
\end{figure}

Consider the birational inverse $\gamma_C^{-1}\colon \Kum {2} {\Jac C}
\dashrightarrow \greg_C$. A general point of the exceptional divisor
$\Xi$ of the Hilbert--Chow morphism $h_C^\circ$ is a nonreduced
subscheme $\xi$ of length~$3$, consisting of a length $2$ subscheme
supported at $L_1\in \Jac C$ and a reduced point $L_2\in\Jac C$.
Taking them general enough, these points are represented in reduced
forms as $L_1 = \cO_C(p_1 + p_2 - 2\infty)$ and $L_2=\cO_C(p_3 + p_4 -
2\infty)$. The condition of defining a point in the generalised Kummer
fourfold of $\Jac C$ is that $L_1^{\otimes 2}\otimes L_2\cong \cO_C$:
that is, $2p_1 + 2p_2 + p_3 + p_4 \sim 6 \infty$. This means that the
cubic interpolation at these four points is bitangent to $C$.

This shows that the rational map $\gamma_C^{-1}$ sends the general
element of $\Xi$ to the big diagonal $\Delta$ of $\Symo 3 {\Sym 2 C}$,
consisting of nonreduced $0$-cycles: in the case above $\xi$~is mapped
to $2[p_1+ p_2] + [p_3 + p_4]$, and so this is the limit point when a
general nonreduced subcheme, defining a cubic interpolation whose
image under $\pi$ consists six different points, goes to $\xi$ (the
point $\xi$ is not sent, say to $[p_1 + p_3] + [p_1 + p_4] + 2p_2$).

The conclusion of this analysis is that $\gamma_C^{-1}$ contracts the
exceptional divisor $\Xi$ to the codimension two locus $\Delta$ in
$\greg_C$.

\subsection{The Galois closure, revisited} \label{ss:galoisrevisited}

Consider the projection to the first four factors $\pr_1\colon C^6 =
C^4\times C^2\to C^4$ and its restriction to $\marc_C$, still denoted
$\pr_1\colon \marc_C\to C^4$.  The quotient $C^4\times C^2\to
C^4\times \Sym 2 C$, restricted to $\marc_C$, induces a double
covering $\marc_C\to \overline \marc_C$, which is the Stein
factorisation of the projection map
\[
pr_1\colon \marc_C\xrightarrow{2:1}\overline{\marc_C}\to C^4.
\]
To see this, observe that any four points $(p_1, \ldots, p_4)\in C^4$ admit
a cubic interpolation, as we
observed in Section~\ref{ss:systemofconics}, so there exist $p_5, p_6$ such that
$p_1+\cdots+p_6\sim 6\infty$. The projection $\pr_1\colon \marc_C\to
C^4$ is thus surjective.  If $p_1+\cdots+p_4\not\sim 4\infty$, then by
Lemma~\ref{lem:equiv_conic} these four points are interpolated by a
unique cubic so the remaining intersection points $p_5, p_6$ with $C$
are uniquely determined. This shows that $(p_1, \ldots, p_4)\in C^4$
has a unique preimage in $\overline{\marc_C}$. If $p_1+\cdots+p_4\sim
4\infty$, by Lemma~\ref{lem:equiv_conic} these four points admit a
conic interpolation and there is a pencil of cubic interpolations
obtained by adding a line joining a point $p\in C$ to the point
$\sigma(p)$. The fibre of $(p_1, \ldots, p_4)\in C^4$ in $\overline
\marc_C$ is thus $C/\latt{\sigma} = \IP^1$.

Consider the composition of the Chow quotient with the Abel--Jacobi
map:
\[
\abel_4\circ \Phi_4\colon C^4 \to \Sym 4 C \to \Jac C
\]
and define the closed subscheme $W\subset C^4$ parametrising
quatruples of points that admit a conic interpolation:
\[
	W\coloneqq (\abel_4\circ \Phi_4)^{-1}\left(\{\cO_C\}\right)=\left\{p\in C^4\mid \Phi_4(p) \sim 4\infty\right\}.
\]
The locus $W$ is exactly where the fibre of the morphism $\overline \marc_C\to C^4$
is $\IP^1$. In fact, it is is simply the blowup of $W$.

\begin{proposition} \label{prop:marc_bar_is_blow}
There is an isomorphism $\overline \marc_C\cong \Bl_W C^4$.
\end{proposition}

\begin{proof}
A similar argument as in \S\ref{ss:degen} and
Proposition~\ref{prop:rank4} shows that
\[
W = \{p\in C^4\mid \rank\rest_{\Phi_4(\xi)}\leq 2\}.
\]
As in the proof of Proposition~\ref{prop:marc}, it follows that $W$ is
locally given by two equations in $C^4$. Consider the two projections
from $\overline \marc_C$:
\[
\xymatrix{ & \marc_C \ar[d]^{2:1}\ar[ddl]_{\pr_1} \ar[ddr]^{\pr_2} & \\
	&\overline \marc_C \ar[dl]^{\pi_1}\ar[dr]_{\pi_2}& \\
C^4 && \Sym 2 C}
\]
Recall from \eqref{eq:aj_exc} that $\Sym 2 C$ contains the exceptional
divisor $E_C= \{p+\sigma(p)\mid p\in C\}$. The inverse ideal sheaf
$\pi_1^{-1}\cI_W\cdot\cO_{\overline \marc_C}$ defines the locus of
those points $(p_1, \ldots, p_4, p_5+p_6)\in C^4\times\Sym 2 C$ such
that $\{p_1,\ldots, p_4\}$ admits a conic interpolation and $\{p_1,
\ldots, p_6\}$ admits a cubic one. As observed above, the only
possibility is that $p_6 = \sigma(p_5)$: that is, $p_5+p_6\in
E_C$. This means that $\pi_1^{-1}\cI_W\cdot\cO_{\overline
  \marc_C}\cong \pi_2^\ast \cI_{E_C}\cong \pi_2^\ast\cO_{\Sym 2
  {C}}(-E_C)$ is an invertible sheaf. By the universal property of
blowup, there exists a morphism $h\colon\overline \marc_C\to \Bl_W
C^4$ factoring $\pi_1$ through the blowup morphism:
\[
\xymatrix{ & \Bl_W C^4\ar[d]\\ \overline \marc_C\ar[ur]^h\ar[r]^{\pi_1} & C^4}
\]
Concretely, every point $p=(p_1, \ldots, p_4, p_5+p_6)\in \overline
\marc_C$ such that $\pi_2(p)\in W$ encodes the equation of a vertical
line $ux-vy = 0$ intersecting $C$ at $p_5$ and $p_6$. The morphism $h$
locally maps $p$ to $(\pi_2(p), (u:v))\in C^4\times \IP^1$. The
inverse morphism is clear. The morphism~$h$ is thus birational and
bijective, and $\overline \marc_C$ is normal since $\marc_C$ is normal by
Proposition~\ref{prop:marc}, so by Zariski's Main Theorem, $h$ is an
isomorphism.
\end{proof}

\begin{remark}
The quotient $\marc_C\to\overline \marc_C$ is the quotient by the
involution $(56)\in H$, so the quotient morphism $\Hat\psi\colon
\marc_C\to \greg_C$ factorises as:
\[
\xymatrix{\marc_C \ar[d]^{\Hat\psi}\ar[r]^-{2:1}& \marc_C/{(56)}=\overline \marc_C \ar[dl] \\
\greg_C = \marc_C/H}
\]
The morphism $\overline \marc_C\to \greg_C$ is not a Galois covering
since $\latt{(56)}$ is not normal in $H$.
\end{remark}

\begin{remark}
The morphism $W\to \IP^2 = |\cC|$ that sends four points admitting a
conic interpolation to the equation of this conic is $24:1$ and it is
ramified when the conic is tangent to~$C$ at one of its Weierstrass
points. At each such point~$w$, the conic interpolation is given by
the tangent line to $C$ at $w$ and a second vertical line intersecting
$C$ at some points $q$ and $\sigma(q)$. The branch locus in $\IP^2$ is
thus a sextic defined by six rational curves, each of them
corresponding to the tangent line to $C$ at a Weierstrass point.
\end{remark}

\section{The branch locus of the covering of the linear system of cubics}\label{s:branch}

\subsection{Computation of the branch locus}\label{ss:branch_locus}

We focus on the $15:1$ covering map $\psi\colon \greg_C\to |\cL|$. By
the Zariski--Nagata purity theorem, the branch scheme $B\coloneqq
\Delta_6\cap |\cL|$ is a divisor in $|\cL| \cong \IP^4$. It
parametrises those cubics that intersect the curve $C$ with at least
one multiple point. Said differently, identifying $|\cL|$ with $\IP
H^0(C, \cL)$ we have
\[
B = \{[s]\in |\cL|\mid \cZ(s)\text { is nonreduced}\},
\]
where $\cZ(s)\subset C$ is the zero scheme of the section $s$ of $\cL$.

\begin{proposition}\label{prop:branch}
The linear system of cubics embeds $C$ in $(\IP^4)^\vee$ and the
branch locus $B\subset \IP^4$ of $\psi$ is the dual variety of $C$, which
is a reduced and irreducible hypersurface of degree $14$.
\end{proposition}

\begin{proof}
Since the line bundle $\cL$ on $C$ is very ample, it defines an
embedding $C\inj |\cL|^\vee\cong (\IP^4)^\vee$ in the dual
projective space. The \emph{conormal variety}~\cite{Kleiman_tangency}
of $C$ for this embedding is
\[
	V_C \coloneqq \{(p, D)\mid T_p C\subset D\}\subset |\cL|^\vee\times |\cL|.
\]
The hyperplanes $D$ in the projective space $|\cL|$ are the cubics in
$\IP_{113}$ and it is easy to check that the condition $T_p C\subset
D$ means that the cubic defined by $D$ is tangent to the curve $C$ at
the point $p$, so in the definition of the branch locus $B$, the
cubics that intersect the curve $C$ with at least one multiple point
correspond here to the hyperplane sections of $|\cL|^\vee$ that are
tangent to the embedding of $C$ in $|\cL|^\vee$. By definition, the
\emph{dual variety} $C^\ast$ of $C$ is the projection of $V_C$ to
$|\cL|$, so with respect to this embedding the branch locus $B$ is the
\emph{dual variety} $C^\ast\subset \IP^4$ of~$C$. It is thus reduced
(we already observed this in the proof of
Proposition~\ref{prop:marc}). Its irreducibility is proved in
~\cite[p.7]{Tevelev}. Its degree can be computed using the general
formula~\cite[Theorem~6.2(i)]{Tevelev}, which reduces in our case to
\[
	\deg B = \int_C c_1(T_C^\vee) + 2 \int_C \cL.
\]
Here $T_C^\vee = K_C$ is a divisor of degree $2$, and $\int_C \cL = 6$, so
$\deg B = 14$ (see also~\cite[Example~10.3]{Tevelev}).
\end{proof}

\begin{remark}
It may be instructive to compute the degree of the branch divisor~$B$
with elementary tools. To do so, we compute the number of intersection
points of $B$ with the pencil of cubics $D_{[a: b]}$ with equation $a
x^3 - b z = 0$, which is the number of points $[a:b]\in \IP^1$ such
that $D_{[a: b]}$ intersects the curve~$C$ with at least one multiple
point. Clearly $[a:b] = [1, 0]$ is a solution that counts with
multiplicity~$4$ since the cubic $x^3 = 0$ cuts $C$ at two points of
multiplicity $3$ (so each point counts twice). For the other solutions
we may put $b = 1$. We may also restrict to the chart $y=1$,
avoiding the point $\infty=[1:0:0]\in C$, because $\infty\in
D_{[0:1]}$ and $D_{[0:1]}\cap C$ consists of the six Weierstrass
points, which are distinct.

Substituting $y=1$ and $z=ax^3$ into the equation \eqref{eq:curve} of
$C$, we find that the points $[x:1:ax^3]$ of $C\cap
D_{[a:1]}$ satisfy an equation of the form $P(x)=0$, where
\[
	P\coloneqq a^2 x^6 - x^5 + \veps_1 x^4 + \cdots + \veps_4 x + \veps_5.
\]
We need the number of values of $a$ such that this polynomial has
repeated roots. This is given by the degree of the discriminant of $P$
as a polynomial in $a$.  The resultant of $P$ and $P'$ is the
following $11\times 11$ determinant:
\[
	\renewcommand*{\arraystretch}{0.3}
	\begin{vmatrix}
		\veps_5 & \veps_4 & \veps_3 & \veps_2 & \veps_1 & -1 & a^2 &&&&\\
		&\ddots &&&&&&\ddots&&&\\
		&&\ddots &&&&&&\ddots&&\\
		&&&\ddots &&&&&&\ddots&\\
		&&&& \veps_5 & \veps_4 & \veps_3 & \veps_2 & \veps_1 & -1 & a^2\\[3ex]
		\veps_4 & 2\veps_3 & 3\veps_2 & 4\veps_1 & -5 & 6a^2 \\
		&\ddots &&&&&\ddots&&&&\\
		&&\ddots &&&&&\ddots&&&\\
		&&&\ddots &&&&&\ddots&&\\
		&&&&\ddots &&&&&\ddots\\
		&&&&&  \veps_4 & 2\veps_3 & 3\veps_2 & 4\veps_1 & -5 & 6a^2
	\end{vmatrix}
\]
This determinant has degree $12$ in $a$, so the discriminant of $P$
has degree $10$. In total we have $14$ intersection points, so the
hypersurface $B$ has degree $14$.
\end{remark}

\begin{corollary}\label{cor:main1}
Let $C$ be a smooth genus two curve. The linear
system of cubics embeds $C$ in $(\IP^4)^\vee$ and the dual variety
$C^\ast\subset \IP^4$ of $C$ is a degree~$14$ irreducible
hypersurface. The second generalised Kummer variety $\Kum 2 {\Jac C}$
of the Jacobian of $C$ is birational to a degree $15$ covering
of~$\IP^4$ branched along~$C^\ast$.
\end{corollary}

This corollary is simply a summary of
Propositions~\ref{prop:GA_normal}, \ref{prop:multiple_4plane} and
\ref{prop:branch}.

\begin{proposition}\label{prop:branch_discrim}
Let $z^2 = f(x)$ be the equation of the curve $C$ in the chart $y=1$
of~$\IP_{113}$, as in Equation~\eqref{eq:curve}.  The branch locus $B$
in $\IP^4$ of coordinates $(\alpha_0:\cdots:\alpha_4)$ has equation:
\[
\frac{1}{\alpha_4^6}\Discr_x\left(\alpha_4^2 f(x) - \left(\alpha_0 x^3
+ \alpha_1 x^2 + \alpha_2 x + \alpha_3\right)^2\right).
\]
\end{proposition}

\begin{proof}
We compute on the affine chart $y=1$ of $\IP_{113}$. The curve $C$ has
equation $z^2 = f(x)$ and we consider a cubic $D$ with equation as in
\eqref{eq:general_cubic}:
\[
	g(x, z) = \alpha_0 x^3 + \alpha_1 x^2 + \alpha_2 x + \alpha_3 + \alpha_4 z.
\]
Let $p = [a:1:b]\in \IP_{113}$. The intersection multiplicity of $C$
and $D$ at $p$ is by definition the dimension as a complex vector
space of the localisation of the quotient $\IC[x, z]/{\latt{ z^2-f(x),
    g(x, z)}}$ at the maximal ideal of the point $p$. It is well known
that if $\alpha_4\neq 0$, this number is the order of vanishing at
$x=a$ of the resultant $R(x):=\Res_z(z^2-f(x),g(x, z))$. In fact
$R(x)= \alpha_4^2 f(x) - \left(\alpha_0 x^3 + \alpha_1 x^2 + \alpha_2
x + \alpha_3\right)^2$. We deduce that $D$ is tangent to $C$ when $R$
has a multiple root, so the branch locus~$B$ is an irreducible
component of the locus of vanishing of the discriminant
$\Discr_x(R)$. Computation shows that $\Discr_x(R)$ has a factor of
$\alpha_4^6$. When $\alpha_4 = 0$, the cubic $D$ consists of three
vertical lines, and $R$ has three double roots that correspond to
tangencies between $C$ and $D$ only when the lines pass through one of
the Weierstrass points. So the factor $\alpha_4^6$ is irrelevant for
the branch locus $B$, hence the result.
\end{proof}

\begin{remark}
An explicit computation of the equation of $B$ is given in
Remark~\ref{magma:branch_locus}.  We proved that the branch locus $B$
is an irreducible component of the locus in $\IP^4$ where the
polynomial $R(x) = \alpha_4^2 f(x) - \left(\alpha_0 x^3 + \alpha_1 x^2
+ \alpha_2 x + \alpha_3\right)^2$ has a multiple root in the variable
$x$. Here $R$ is not monic but we may simply consider $\bar R(x) =
\alpha_4^2 f(x) - \left( x^3 + \alpha_1 x^2 + \alpha_2 x +
\alpha_3\right)^2$, compute the discriminant of $R$, divide by
$\alpha_4^6$ and homogeneise with respect to the variable $\alpha_0$
to recovering the branch locus $B$. Following Arnold
(see~\cite{Napolitano} and references therein), the branch locus is
stratified in closed subschemes $B^\lambda$, where $\lambda$ is a
partition of the integer $6$. The main strata are the caustic stratum
$B^{3, 1, 1, 1}$ and the Maxwell stratum $B^{2, 2, 1, 1}$.
\end{remark}

Recalling the divisors $R_1, R_2$ defined in
Equations~\eqref{div:R1} and \eqref{div:R2}, that describe the
ramification of the morphism $\phi$, let us introduce the following
divisors on $\greg_C$:
\begin{align}
	R_i^0\coloneqq R_i\cap \greg_C, \quad i=1, 2.
\end{align}
We denote by $H$ the pullback $H := \psi ^\ast L$, where $L \subset
\IP^4$ is a hyperplane.
Since $\psi$ is finite and $L$ is (very) ample, the
divisor $H$ is ample.  From the properties of the covering $\psi\colon
\greg_C\to\IP^4$, we deduce some useful geometric information on
$\greg_C$:

\begin{corollary}\label{cor:divisors_greg}\text{}
	\begin{enumerate}
		\item  The  divisor $R_2^\circ$ is very ample.
		\item $14H= R_1^\circ + 2R_2^\circ$.
		\item The canonical divisor  is $K_{\greg_C} = -5H + R_2^\circ$.
	\end{enumerate}
\end{corollary}

\begin{proof}
By Proposition~\ref{prop:multiple_4plane}, the morphism $\psi$ is
Gorenstein so we can apply the general theory
of~\cite[Theorem~2.1(ii)]{CE}: the Tschirnhausen bundle $\cE^\vee$ of
$\psi$ gives an embedding $j\colon\greg_C\inj \IP(\cE)$
such that the ramification divisor $R_2^\circ$ of $\psi$ satisfies
\[
	\cO_{\greg_C}(R_2^\circ)\cong \omega_{\greg_C/\IP^4}\cong \cO_{\greg_C}(1)\coloneqq j^\ast\cO_{\IP(\cE)}(1).
\]
It follows that $R_2^\circ$ is very ample.  By
Lemma~\ref{lem:ram_phi}, we have $\psi^\ast B = R_1^\circ +
2R_2^\circ$ and by Proposition~\ref{prop:branch} we have $B = 14L$, so
$14H = R_1^\circ + 2R_2^\circ$.  By Lemma~\ref{lem:ram_phi} again,
$\psi$~has simple ramification along $R_2^\circ$, so $K_{\greg_C} =
\psi^\ast K_{\IP^4} + R_2^\circ = -5H + R_2^\circ$.
\end{proof}

\begin{remark}
The covering $\psi\colon \greg_C\to\IP^4$ satisfies some of the
assumptions defining a \emph{general} multiple space
in~\cite[Definition~2.2]{FPV}. It would thus be interesting to know
whether $R_2^\circ$ is nonsingular and whether the restriction
$\psi_{|R_2^\circ}\colon R_2^\circ\to B$ is the normalisation map.
This covering is not Galois and has relatively high degree, making
it hard to understand its possible deformations. Moreover, we noted in
Remark~\ref{rem:non_cartier} that $\greg_C$ is not a local complete
intersection, so the relative cotangent complex of $\psi$ is not
perfect, making it harder to compute the deformations of the covering
$\psi$. Our interest in proving Proposition~\ref{prop:marc} is that
the Galois closure $\Psi\colon\marc_C\to\IP^4$ of the covering has
perfect cotangent complex, so the study of its deformations should
behave more nicely: this is work in progress.
\end{remark}

\appendix

\section{Moduli spaces of polarised IHS manifolds}\label{app:moduli_spaces}

\subsection{The moduli spaces $\cM_{\KK}^{d}$:} \label{ss:K3moduli}
they parametrise degree $d$ polarised K3 surfaces, and their dimension
is~$19$.  One easy example is the projective family of smooth quartic
surfaces in $\IP^3$, which is also $19$-dimensional. Since the moduli
space is irreducible, this shows that $\cM_\KK^4$ is unirational. The
known results can be summarised as follows: the moduli spaces
$\cM_\KK^{2e}$ are unirational for $1\leq e\leq 12$ and $e = 15, 16,
17, 19$ (see~\cite[\S 4]{GHS_moduliK3} and references
therein). Unirationality properties can be also obtained for moduli of
lattice-polarised K3 surfaces, see for
instance~\cite[Proposition~3.9]{Roulleau}. In the other direction, it
was shown in \cite{GHS_K3} that $\cM_\KK^{2e}$ has non-negative
Kodaira dimension if $e\ge 40$, with four possible exceptions, and is
of general type for $e\ge 62$ and a few smaller numbers.

\subsection{The moduli spaces $\cM_{\hilb^n}^{d, \gamma}$:} \label{ss:Hilbmoduli}
they parametrise polarised IHS manifolds of Hilbert type, \ie
deformation equivalent to the Hilbert scheme of $n$ points on a K3
surface, with $n\geq 2$, of degree $d$ and divisibility $\gamma$.
Their dimension is $20$, whereas the families of Hilbert schemes of
$n$ points on polarised $K3$ surfaces have dimension~$19$.
Gritsenko, Hulek and Sankaran~\cite[Theorem~4.1]{GHS_moduliIHS} proved
that $\cM_{\hilb^2}^{2e, 1}$ is of general type if $e\geq
12$. Otherwise:
\begin{itemize}
	\item $\cM_{\hilb^2}^{2, 2}$ is unirational since it contains
          a $20$-dimensional family using double coverings of
          EPW-sextics (see~O'Grady~\cite{OG_EPW}
          and~\cite[Example~4.3]{GHS_moduliK3}).

	\item $\cM_{\hilb^2}^{6, 1}$ is unirational since it contains
          a $20$-dimensional family using Fano varieties of lines on
          cubic fourfolds (see~Beauville and Donagi~\cite{BD}
          and~\cite[Example~4.2]{GHS_moduliK3}).

	\item $\cM_{\hilb^2}^{38, 2}$ is unirational:
          Iliev and Ranestad~\cite{IR} constructed a $20$-dimensional
          family (see~\cite[Example~4.4]{GHS_moduliK3}
          and~\cite[Proposition~1.4.16]{MongardiPhD}).
          
	\item$\cM_{\hilb^2}^{22, 2}$ is unirational:
          Debarre and Voisin~\cite{DV} constructed a $20$-dimensional
          family (see~\cite[Example~4.5]{GHS_moduliK3}).

	\item $\cM_{\hilb^3}^{4, 2}$ is unirational:
          Iliev, G.~Kapustka, M.~Kapustka and Ranestad~\cite{IKKR} constructed
          a $20$-dimensional family called \emph{EPW cubes}.
\end{itemize}

\subsection{The moduli spaces $\cM_{\kum^n}^{d, \gamma}$:} \label{ss:Kummermoduli}
they parametrise polarised IHS manifolds of Kummer type, \ie
deformation equivalent to the $n$-th generalised Kummer variety of an
abelian surface, with $n\geq 2$, of degree $d$ and divisibility
$\gamma$. Their dimension is $4$, whereas the families of polarised
abelian surfaces have dimension $3$. By
Dawes~\cite[Theorem~3.6]{Dawes} we know that $\cM_{\kum^2}^{2d, 1}$ is
of general type if $d \gg 0$. See also \cite{Dawes_orig}. Otherwise:
\begin{itemize}
	\item $\cM_{\kum^n}^{2, 1}$ is uniruled if $n\geq 15$ or
          $n=17, 20$ (see~\cite[Theorem~7.5]{BBFW}).

	\item $\cM_{\kum^n}^{2, 2}$ is uniruled if $n=4t-2$ with
          $t\leq 11$ or $t=13, 15, 17, 19$
          (see~\cite[Theorem~7.5]{BBFW}).

	\item $\cM_{\kum^2}^{2, 2}$ is rational
          (see~\cite[Theorem~5.4]{WW} and~\cite[Theorem~7.6]{BBFW}).

\end{itemize}

\subsection{The moduli spaces $\cM_{\ogsix}^{d, \gamma}$:} \label{ss:OG6moduli}
these $6$-dimensional spaces parametrise polarised IHS manifolds of
type OG6, \ie deformation equivalent to an O'Grady IHS sixfold, of
degree $d$ and divisibility
$\gamma$. Following~\cite[Theorem~7.2]{BBFW} we have:
\begin{itemize}
	\item $\cM_{\ogsix}^{2d, 1}$ is uniruled if $d\leq 12$;
	\item $\cM_{\ogsix}^{4t-1, 2}$ is uniruled if $t\leq 10$ or $t=12$;
	\item $\cM_{\ogsix}^{4t-2, 2}$ is uniruled if $t\leq 9$ or $t=11, 13$.
	\item The moduli space $\cM_{\ogsix}^{6, 2}$ and
          $\cM_{\ogsix}^{2, 1}$ are unirational
          (see~\cite[Theorem~4.9]{WW} and~\cite[Theorem~7.6]{BBFW}).
\end{itemize}

\subsection{The moduli spaces $\cM_{\ogten}^{d,\gamma}$:} \label{ss:OG10moduli}
these $21$-dimensional spaces parametrise polarised IHS manifolds of
type OG10, \ie deformation equivalent to an O'Grady IHS tenfold, of
degree $d$ and divisibility $\gamma$. As far as the authors know, no
unirationality result is known for these spaces. It is conjectured in 
\cite{GHS_OG10} that $\cM_{\ogten}^{24,3}$, $\cM_{\ogten}^{60,3}$ and
$\cM_{\ogten}^{96,3}$ are uniruled, and it is shown that
$\cM_{\ogten}^{d,1}$ is of general type unless $d$ is a power of~$2$.

\section{Alternative views on the proof of Proposition~\ref{prop:birational_F1}}
\label{app:complements_prop_birational_F1}

\subsection{Saturation}\label{ss:saturation}

In this proof, we are interested in the locus inside $\Tilde F_1$
where the rational function ${\tilde a}_2$ can be extended, keeping
the Cramer relation true. In the polynomial ring $\IC[\bx_1, \bx_2,
  \bw_1, \bw_2, \bz_3, \tilde{a}_2]$, we consider the ideal $I$ defining
the Cramer relation satisfied by the rational function $\tilde{a}_2$
and the ideal $J$ defining the divisor $F_1$. The locus inside $F_1$
where $\tilde{a}_2$ extends is the intersection with the Zariski
closure $\overline{\cV(I)\setminus \cV(J)}\cap \cV(J)$.

It is a classical result that this is the zero locus of the
saturation $(I\colon J^\infty) + J$ of the ideal $I$ with respect to $J$.
A Gr\"obner basis computation (see Remark~\ref{magma:birational})
gives
\[
(I\colon J^\infty) + J = \latt{ \bw_1\bw_2^2(\bw_1-\bw_2)}.
\]
We recover the locus $G$ defined in~\eqref{eq:closure_graph_a2}.

\subsection{Computations in the chart $\cU_{(2, 1)}$} \label{ss:localcomputation21}

Following the notation of the proof, in the chart $\cU_{(2, 1)}$ the
coordinate functions are $a_0, a_1, a_2, b_0, b_1, b_2, c_0, c_1, c_2$
with the relations given in the proof. To compute them in terms of a
triple of points $(\bx_1, \by_1),(\bx_2, \by_2),(\bx_3, \by_3)$ we use
the generators of the ideal $I_{(a, b, c)}$. The generator $\bx^2 -
a_0 - a_1\bx - a_2 \by$ means that:
\[
\begin{pmatrix}
	1 & \bx_1 & \by_1\\
	1 & \bx_2 & \by_2\\
	1 & \bx_3 & \by_3\\
\end{pmatrix}
\begin{pmatrix}
	a_0\\a_1\\a_2
\end{pmatrix}
=
\begin{pmatrix}
	\bx_1\\ \bx_2 \\ \bx_3
\end{pmatrix},
\]
so the coordinates $a_i$ are given by Cramer's rule:
\[
a_0 = \frac
{\begin{vmatrix}
		\bx_1^2 & \bx_1 & \by_1\\
		\bx_2^2 & \bx_2 & \by_2\\
		\bx_3^2 & \bx_3 & \by_3
\end{vmatrix}}
{\begin{vmatrix}
		1 & \bx_1 & \by_1\\
		1 & \bx_2 & \by_2\\
		1 & \bx_3 & \by_3
\end{vmatrix}},\quad
a_1 = \frac
{\begin{vmatrix}
		1 & \bx_1^2 & \by_1\\
		1 & \bx_2^2 & \by_2\\
		1 & \bx_3^2 & \by_3
\end{vmatrix}}
{\begin{vmatrix}
		1 & \bx_1 & \by_1\\
		1 & \bx_2 & \by_2\\
		1 & \bx_3 & \by_3
	\end{vmatrix}
},\quad
a_2 = \frac{\begin{vmatrix}
		1 & \bx_1 & \bx_1^2\\
		1 & \bx_2 & \bx_2^2\\
		1 & \bx_3 & \bx_3^2
	\end{vmatrix}
}
{\begin{vmatrix}
		1 & \bx_1 & \by_1\\
		1 & \bx_2 & \by_2\\
		1 & \bx_3 & \by_3
	\end{vmatrix}
},
\]
and similarly
\[
b_0 = \frac
{\begin{vmatrix}
		\bx_1\by_1 & \bx_1 & \by_1\\
		\bx_2\by_2 & \bx_2 & \by_2\\
		\bx_3\by_3 & \bx_3 & \by_3
\end{vmatrix}}
{\begin{vmatrix}
		1 & \bx_1 & \by_1\\
		1 & \bx_2 & \by_2\\
		1 & \bx_3 & \by_3
\end{vmatrix}},\quad
b_1 = \frac
{\begin{vmatrix}
		1 & \bx_1\by_1 & \by_1\\
		1 & \bx_2\by_2 & \by_2\\
		1 & \bx_3\by_3 & \by_3
\end{vmatrix}}
{\begin{vmatrix}
		1 & \bx_1 & \by_1\\
		1 & \bx_2 & \by_2\\
		1 & \bx_3 & \by_3
	\end{vmatrix}
},\quad
b_2 = \frac{\begin{vmatrix}
		1 & \bx_1 & \bx_1\by_1\\
		1 & \bx_2 & \bx_2\by_2\\
		1 & \bx_3 & \bx_3\by_3
	\end{vmatrix}
}
{\begin{vmatrix}
		1 & \bx_1 & \by_1\\
		1 & \bx_2 & \by_2\\
		1 & \bx_3 & \by_3
	\end{vmatrix}
},
\]
and
\[
c_0 = \frac
{\begin{vmatrix}
		\by_1^2 & \bx_1 & \by_1\\
		\by_2^2 & \bx_2 & \by_2\\
		\by_3^2 & \bx_3 & \by_3
\end{vmatrix}}
{\begin{vmatrix}
		1 & \bx_1 & \by_1\\
		1 & \bx_2 & \by_2\\
		1 & \bx_3 & \by_3
\end{vmatrix}},\quad
c_1 = \frac
{\begin{vmatrix}
		1 & \by_1^2 & \by_1\\
		1 & \by_2^2 & \by_2\\
		1 & \by_3^2 & \by_3
\end{vmatrix}}
{\begin{vmatrix}
		1 & \bx_1 & \by_1\\
		1 & \bx_2 & \by_2\\
		1 & \bx_3 & \by_3
	\end{vmatrix}
},\quad
c_2 = \frac{\begin{vmatrix}
		1 & \bx_1 & \by_1^2\\
		1 & \bx_2 & \by_2^2\\
		1 & \bx_3 & \by_3^2
	\end{vmatrix}
}
{\begin{vmatrix}
		1 & \bx_1 & \by_1\\
		1 & \bx_2 & \by_2\\
		1 & \bx_3 & \by_3
	\end{vmatrix}
}.
\]
The relations three between these nine coordinates, given in the
proof, are easy to check.  Substituting as in the proof, we get the
following formulas:
\begin{align*}
	\tilde{a}_1&=\frac{\bx_1A_1(\bw_1, \bw_2, \bz_3)}{\bw_1 \bx_2^2(\bw_1-\bw_2)},
        &\qquad
	\tilde{a}_2&=\frac{\bx_1A_2(\bw_1, \bw_2, \bz_3)}{\bw_1 \bx_2^2(\bw_1-\bw_2)},\\
	\tilde{b}_1&=\frac{\bx_1B_1(\bw_1, \bw_2, \bz_3)}{\bw_1 \bx_2^2(\bw_1-\bw_2)},
        &\qquad
	\tilde{b}_2&=\frac{\bx_1B_2(\bw_1, \bw_2, \bz_3)}{\bw_1 \bx_2^2(\bw_1-\bw_2)},\\
	\tilde{c}_1&=\frac{\bx_1C_1(\bw_1, \bw_2, \bz_3)}{\bw_1 \bx_2^2(\bw_1-\bw_2)},
        &\qquad
	\tilde{c}_2&=\frac{\bx_1C_2(\bw_1, \bw_2, \bz_3)}{\bw_1 \bx_2^2(\bw_1-\bw_2)},
\end{align*}
where $A_i, B_i, C_i$ are polynomial expressions. This shows that all
the coordinate functions vanish at $\bx_1=0$ when $\bw_1
\bx_2^2(\bw_1-\bw_2)\neq 0$, so the generic point of the divisor $F_1$ is
sent to the ideal $ I_\infty = \latt{ \bx^2, \bx\by, \by^2}$. This
gives a different proof that $\gamma_A$ contracts the divisor $F_1$ to
the point $Z_\infty$.

\subsection{The projective embedding}\label{s:grothendiek}

We use an explicit projective embedding of $\Hilb 3  {\IC^2}$ following the
presentation given by Haiman~\cite{Haiman} of the original and general
construction due to Grothendieck. We first recall this construction.
Let $M$ be the set of monomials in the variables $\bx, \by$ of degree
at most $3$. For any ideal $I\in \Hilb 3 {\IC^2}$, by Gordan~\cite{Gordan}
the quotient $\IC[\bx, \by]/I$ is generated by $M$ (at this point,
monomials of degree at most two would suffice, but we need degree
three monomials for the projective embedding). Denote by $V\coloneqq
\Span(M)\subset \IC[\bx, \by]$ the vector subspace generated by
$M$. For any $I\in \Hilb 3  {\IC^2}$, the linear map
\[
\pi_I\colon V\to \IC[\bx, \by]/I
\]
is surjective, and its kernel $\ker(\pi_I)$ has codimension three in
$V$. Instead of working with a basis of this kernel, it is more
convenient to work with its equations, so we consider its annihilator
$\ker(\pi_I)^\perp\subset V^\ast$,
which has dimension three. By a result of Grothendieck, we get an
embedding in the Grassmannian of $3$-dimensional subspaces of
$V^\ast$:
\[
\Hilb 3  {\IC^2} \inj \Grass(3, V^\ast), \quad I\mapsto \ker(\pi_I)^\perp.
\]
The projective embedding of $\Hilb 3  {\IC^2}$ inside which we will study the
behaviour of the map $g$ is the Pl\"ucker embedding (where we use here the projective space of lines in $\Ext^3 V^\ast$):
\[
\wp\colon \Hilb 3  {\IC^2} \inj \IP \left(\Ext^3 V^\ast\right),
\quad I\mapsto \Ext^3\left(\ker(\pi_I)^\perp\right).
\]

Let $\cE\subset \Hilb 3  {\IC^2}$ be the exceptional divisor, parametrising
non-reduced subschemes, and $\Delta\subset \Sym 3  {\IC^2}$ the big
diagonal.  We put $A =  {\IC^2}$ for more readibility. The Hilbert--Chow morphism $h_A$ restricts to an
isomorphism between the open subsets $\Hilb 3 A\setminus \cE$ and
$\Sym 3  A\setminus \Delta$. The rational map from $\Tilde A^3$ to
$\IP \left(\Ext^3 V^\ast\right)$ is regular on the following open
subsets:
\[
\Tilde A^3\setminus \Delta'' \xrightarrow{\pi_{\Tilde A}} \Sym 3
         {\Tilde A} \setminus \Delta' \xrightarrow{\Sym 3 {\bl_A}}
         \Sym 3 A\setminus \Delta \xrightarrow{h_A^{-1}} \Hilb 3
         A\setminus \cE\xrightarrow{\wp} \IP \left(\Ext^3
         V^\ast\right)
\]
where $\Delta'\coloneqq \left(\Sym 3 {\bl_A}\right)^{-1}(\Delta)$ and
$\Delta''\coloneqq (\pi_{\Tilde A})^{-1}(\Delta')$. After restricting
to the fibres over the origin, we recover the restriction of the
rational map $\wp\circ g$, and we see that it is regular on an open
subset containing $(\Tilde A^3\setminus \Delta'')\cap \Tilde A^3_0$,
but this subset is not optimal since it contains divisors. To compute
the image of the divisor $F_1$ contracted by~$\gamma_A$, we compute
for each partition $\lambda$ of the integer $3$, the morphism
\[
\wp\circ g\colon g^{-1}\left(\Kum 2 A \cap \cU_\lambda\cap\left(\Hilb
3 A\setminus \cE\right)\right)\to \IP \left(\Ext^3 V^\ast\right).
\]

Let us compute on the chart $\cU_{(1, 1, 1)}$.  The vector space $V$
has basis
\[
(1, \bx, \by, \bx^2, \bx\by, \by^2, \bx^3, \bx^2\by, \bx\by^2, \by^3).
\]
For any ideal $I\coloneqq I_{(e,a)}\in \cU_{(1, 1, 1)}$
generated by:
\[
I_{(e, a)}\coloneqq \latt{ \bx^3-e_1\bx^2+e_2\bx-e_3, \by-(a_0+a_1\bx + a_2\bx^2)},
\]
the quotient space $\IC[\bx, \by]/I$ has basis $(1, \bx, \bx^2)$
modulo $I$. The morphism $\pi_I$ defines a $3\times 10$ matrix 
whose coefficients depend on $e$ and $a$. The kernel of this matrix is
organised as a $(7\times 10)$-matrix  whose rows are the
coordinates of the generators of $\ker(\pi_I)$. Interpreting duality
as a canonical scalar product, the kernel of this second matrix is organised as a
$(3\times 10)$-matrix whose rows are the coefficients of the equations
of $\ker(\pi_I)^\perp$ in the dual basis of $V^\ast$. We restrict to
$\Kum 2 A$ by inserting the equations of $\Kum 2 A$ in $\Hilb 3 A$ in
our chart, that is $e_1=0, a_0 = \frac{2}{3}a_2e_2$, and we arrive at the
following matrix:
\begin{align*}
C &=\frac{1}{27} \left(\begin{matrix}
0 & 0 & 27a_2 & 27 & 27a_1 & 9a_2^2 e_2 + 27a_1^2 & 0 \\
0& 27& 27a_1& 0& -9a_2e_2& -18a_1a_2e_2+27a_2^2e_3& -27e_2\\
27& 0& 18a_2e_2& 0& 27a_2e_3& 12a_2^2e_2^2+54a_1a_2e_3& 27e_3\\
\end{matrix}\right.\cdots\\
&\cdots\left.
\begin{matrix}
 -9a_2e_2  & -18a_1 a_2 e_2 + 27 a_2^2e_3 \\
 -27a_1e_2+27a_2e_3& 3a_2^2e_2^2-27a_1^2e_2+54a_1a_2e_3\\
 27a_1e_3& 9a_2^2e_2e_3+27a_1^2e_3\\
 \end{matrix}\right.\cdots\\
&\cdots\left.
\begin{matrix}
	9a_2^3e_2^2-27a_1^2a_2e_2+81a_1a_2^2e_3\\
	9a_1a_2^2e_2^2-27a_1^3e_2+81a_1^2a_2e_3\\
	27a_1a_2^2e_2e_3+8a_2^3e_2^3+27a_1^3e_3+27a_2^3e_3^2
\end{matrix}\right)
\end{align*}
We now use formulas~\eqref{eq:cramer} and~\eqref{eq:model} to express
the $120$ Pl\"ucker coordinates~$p_{i,j,k}$ of~$I$ in~$\IP\left(\Ext^3
V^\ast\right)$ as rational functions of the variables $\bx_1, \bw_1,
\bx_2, \bw_2, \bz_3$. Those are all the $(3\times 3)$ minors of the
matrix $C$. After some simplifications, we see that $p_{1, 2, 3}$ is
divisible by $\bx_1^{26}$ and that all the others are divisible by
$\bx_1^{27}$. We obtain the rational image of the divisor $F_1$ by
putting $\bx_1=0$, which gives the point $[1:0:\ldots:0]\in
\IP^{120}$. A similar method shows that these are the Pl\"ucker
coordinates of the point $Z_\infty$.  The computation on the chart
$\cU_{(2, 1)}$ is similar.

\section{The scripts used in this work} \label{s:scripts}

We reproduce below the scripts used in this paper. None of our proofs
actually needed computer algebra tools: these only served as a
guidance. We used \Macaulay\cite{Macaulay2} and \Magma\cite{Magma}.

\begin{remark}\label{macaulay:lci}

Here is a \Macaulay script used in Remark~\ref{rem:non_cartier} to
check that the variety $\greg_A$ is not a local complete intersection
scheme:
	\begin{verbatim}
		loadPackage "InvariantRing"
		R = QQ [x1, x2, x3, y1, y2, y3]
		M12 = matrix{{0, 1, 0, 0, 0, 0},
			{1, 0, 0, 0, 0, 0},
			{0, 0, 1, 0, 0, 0},
			{0, 0, 0, 0, 1, 0},
			{0, 0, 0, 1, 0, 0},
			{0, 0, 0, 0, 0, 1}}
		M123 = matrix{{0, 0, 1, 0, 0, 0},
			{1, 0, 0, 0, 0, 0},
			{0, 1, 0, 0, 0, 0},
			{0, 0, 0, 0, 0, 1},
			{0, 0, 0, 1, 0, 0},
			{0, 0, 0, 0, 1, 0}}
		L = {M12, M123}
		S3 = finiteAction(L,R)
		g = invariants S3
		netList g
		A = QQ[x1, x2, x3, y1, y2, y3, f0, f1, f2, f3, f4, f5, f6, f7, f8]
		inv = {substitute(g#0, A), substitute(g#1, A), substitute(g#2, A),
		   	   substitute(g#3, A), substitute(g#4, A), substitute(g#5, A),
		   	   substitute(g#6, A), substitute(g#7, A), substitute(g#8, A)}
		I = ideal{f0 - inv#0, f1 - inv#1, f2 - inv#2, f3 - inv#3, f4 - inv#4,
			          f5 - inv#5, f6 - inv#6, f7 - inv#7, f8 - inv#8}
		loadPackage "Elimination"
		J = eliminate({x1, x2, x3, y1, y2, y3}, I)
		-- Computation of Sym^3_0 A : the equations are f0, f1
		sym = J + ideal{f0, f1}
		loadPackage "TorAlgebra"
		isGorenstein sym
		isCI sym
		-- Computation of Sym^3_0 hat A : the equations are f1, f3
		model = J + ideal{f1, f3}
		isGorenstein model
		isCI model

	\end{verbatim}
\end{remark}

\begin{remark}\label{magma:galois}
Here is a \Magma script to compute the group $H$ defined in \S\ref{ss:group_H}:
	\begin{verbatim}
		G := SymmetricGroup(6);
		H := sub<G | [(1, 2), (1, 3)(2, 4), (1, 5)(2, 6)]>;
		Order(H);
		IsNormal(G, H);
	\end{verbatim}
\end{remark}

\begin{remark}\label{magma:branch_locus}
Here is a \Magma script to compute the branch locus~$B$ in \S\ref{ss:branch_locus}:
	\begin{verbatim}
		R <e1, e2, e3, e4, e5> := PolynomialRing(Rationals(), 5);
		A <x, a0, a1, a2, a3, a4> := PolynomialRing(R, 6);
		f := x ^ 5 - e1 * x ^ 4 - e2 * x ^ 3 - e3 * x ^ 2 - e4 * x - e5;
		p := a0 * x ^ 3 + a1 * x ^ 2 + a2 * x + a3;
		r :=  a4 ^2 * f - p ^ 2;
		b := Discriminant(r, x) div a4^6;
		Degree(b);
	\end{verbatim}
\end{remark}

\begin{remark}\label{magma:birational}
Here is the \Magma script used in \S\ref{ss:saturation}:
\begin{verbatim}
	A<a2,x1,x2,w1,w2,z3> := PolynomialRing(Rationals(), 6, "elim", 1);
	x3 := - x1 - x2;
	v1 := w1 + z3;
	v2 := w2 + z3;
	y1 := x1 * v1;
	y2 := x2 * v2;
	y3 := x3 * z3;
	D := (x1-x2) * (x1 - x3) * (x2 - x3);
	I := ideal<A | x1 * w1 + x2 * w2,
	     D * a2 - (x2 - x3) * y1 - (x3 - x1) * y2	- (x1 - x2) * y3>;
	J := ideal<A | x1, x2>;
	C := Saturation(I, J) + J;
	GroebnerBasis(C);
\end{verbatim}
\end{remark}

\bibliographystyle{plainurl}
\bibliography{Biblio}

\end{document}